\pdfoutput=1
\documentclass[12pt]{paper}

\usepackage{color}
\usepackage{mathptmx}       
\usepackage{helvet}         
\usepackage{courier}        
\usepackage{type1cm}        
\usepackage{makeidx}         
\usepackage{graphicx}        
\usepackage{multicol}        
\usepackage[bottom]{footmisc}
\usepackage{amsmath}
\usepackage{amsfonts}
\usepackage{amssymb}
\usepackage{amsthm}
\usepackage{amsopn}
\usepackage{subcaption}
\usepackage{comment}
\usepackage{soul,xcolor}
\usepackage{hyperref}

\newtheorem{theorem}{Theorem}
\newtheorem{corollary}[theorem]{Corollary}

\newtheorem{lemma}[theorem]{Lemma}

\newtheorem{proposition}[theorem]{Proposition}

\newtheorem{remark}[theorem]{Remark}
\newtheorem{definition}[theorem]{Definition}

\usepackage[top=2cm, bottom=2cm, left=2cm, right=2cm]{geometry}
\numberwithin{theorem}{section}
\numberwithin{figure}{section}
\numberwithin{equation}{section}

\DeclareMathOperator{\dist}{dist}
\DeclareMathOperator{\SLE}{SLE}

\DeclareMathOperator{\diam}{diam}

\newcommand{\EE}{\mathbb{E}}

\newcommand{\pa}{\partial}

\newcommand{\no}{\noindent}

\newcommand{\BGE}{\begin{equation}}
\newcommand{\BGEN}{\begin{equation*}}
\newcommand{\EDE}{\end{equation}}
\newcommand{\EDEN}{\end{equation*}}

\def\sem{\setminus}
\def\lin{\overline}

\DeclareMathOperator{\hcap}{hcap} \DeclareMathOperator{\id}{id}

 \DeclareMathOperator{\doub}{doub}

\DeclareMathOperator{\hm}{hm}

\begin{document}

\title{Boundary Arm Exponents for SLE}
\author{Hao Wu and Dapeng Zhan}
\date{\today}


%
%
\maketitle

\abstract{We derive boundary arm exponents for SLE. Combining with the convergence of critical lattice models to SLE, these exponents would give the  alternating half-plane arm exponents for the corresponding lattice models.}
\newcommand{\eps}{\epsilon}
\newcommand{\ov}{\overline}
\newcommand{\U}{\mathbb{U}}
\newcommand{\T}{\mathbb{T}}
\newcommand{\HH}{\mathbb{H}}
\newcommand{\LA}{\mathcal{A}}
\newcommand{\LC}{\mathcal{C}}
\newcommand{\LD}{\mathcal{D}}
\newcommand{\LF}{\mathcal{F}}
\newcommand{\LK}{\mathcal{K}}
\newcommand{\LE}{\mathcal{E}}
\newcommand{\LL}{\mathcal{L}}
\newcommand{\LU}{\mathcal{U}}
\newcommand{\LV}{\mathcal{V}}
\newcommand{\LZ}{\mathcal{Z}}
\newcommand{\LH}{\mathcal{H}}
\newcommand{\R}{\mathbb{R}}
\newcommand{\C}{\mathbb{C}}
\newcommand{\N}{\mathbb{N}}
\newcommand{\Z}{\mathbb{Z}}
\newcommand{\E}{\mathbb{E}}
\newcommand{\PP}{\mathbb{P}}
\newcommand{\QQ}{\mathbb{Q}}
\newcommand{\A}{\mathbb{A}}
\newcommand{\bn}{\mathbf{n}}
\newcommand{\MR}{MR}
\newcommand{\cond}{\,|\,}
\newcommand{\la}{\langle}
\newcommand{\ra}{\rangle}
\newcommand{\tree}{\Upsilon}
\setstcolor{blue}

\section{Introduction}\label{sec::introduction}
Schramm-Loewner evolution (SLE) was introduced by Oded Schramm \cite{SchrammFirstSLE} as the candidates for the scaling limits of interfaces in 2D critical lattice models. It is a one-parameter family of random fractal curves in simply connected domains from one boundary point to another boundary point, which is indexed by a positive real $\kappa$.
Since its introduction, it has been proved to be the limits of several lattice models:
$\SLE_2$ is the limit of Loop Erased Random Walk and $\SLE_8$ is the limit of the Peano curve of Uniform Spanning Tree \cite{LawlerSchrammWernerLERWUST}, $\SLE_3$ is the limit of the interface in critical Ising model and $\SLE_{16/3}$ is the limit of the interface in FK-Ising model \cite{CDCHKSConvergenceIsingSLE}, $\SLE_4$ is the limit of the level line of discrete Gaussian Free Field \cite{SchrammSheffieldDiscreteGFF} and $\SLE_6$ is the limit of the interface in critical Percolation \cite{SmirnovPercolationConformalInvariance}.

In the study of lattice models, arm exponents play an important role. Take percolation for instance, Kesten has shown that \cite{KestenScalingRelationPercolation} in order to understand the behavior of percolation near its critical point, it is sufficient to study what happens at the critical point, and many results would follow from the existence and values of the arm exponents. To be more precise, consider critical percolation with fixed mesh equal to 1, and for $n\ge 2$, consider the the event $E_n(z, r, R)$ that there exist $n$ disjoint crossings of the annulus $A_z(r, R):=\{w\in\C: r<|w-z|<R\}$, not all of the same color. People would like to understand the decaying of the probability of $E_n(z, r,R)$ as $R\to\infty$. It turns out that this probability decays like a power in $R$, and the exponent is called plane arm exponents.
There are another related quantities, called half-plane arm exponents. In this case, consider critical percolation in the upper-half plane $\HH$, and for $n\ge 1, x\in\R$, define $H_n(x, r, R)$ to be the event that there exist $n$ disjoint crossings of the semi-annulus $A_x^+(r, R):=\{w\in\HH: r<|w-x|<R\}$.
After the identification between $\SLE_6$ and the limit of critical percolation on triangular lattice \cite{SmirnovPercolationConformalInvariance}, one could derive these exponents via the corresponding arm exponents for $\SLE_6$ \cite{SmirnovWernerCriticalExponents}:
\[\PP\left[E_n(z, r,R)\right]= R^{-\alpha_n+o(1)},\quad \PP\left[H_n(x, r,R)\right]=R^{-\alpha_n^++o(1)},\quad \text{as }R\to\infty,\]
where
\[\alpha_n:=(n^2-1)/12,\quad \alpha_n^+:=n(n+1)/6.\]

In this paper, we derive boundary arm exponents for $\SLE_{\kappa}$. Combining with the identification between the limit of critical lattice model and $\SLE$ curves, these exponents for $\SLE$ would imply the arm exponents for the corresponding lattice models.

Fix $\kappa>4$ and let $\eta$ be an $\SLE_{\kappa}$ in $\HH$ from 0 to $\infty$. Suppose that $y\le 0<\eps\le x$ and let $T$ be the first time that $\eta$ swallows the point $x$ which is almost surely finite when $\kappa>4$. We first define the crossing event $H_{2n-1}$ (resp. $\hat{H}_{2n}$) that $\eta$ crosses between the ball $B(x,\eps)$ and the half-infinite line $(-\infty, y)$ at least $2n-1$ times (resp. at least $2n$ times) for $n\ge 1$. To be precise with the definition, we need to introduce a sequence of stopping times.
Set $\tau_0=\sigma_0=0$. Let $\tau_1$ be the first time that $\eta$ hits the ball $B(x,\eps)$ and let $\sigma_1$ be the first time after $\tau_1$ that $\eta$ hits $(-\infty, y)$. For $n\ge 1$, let $\tau_n$ be the first time after $\sigma_{n-1}$ that $\eta$ hits the connected component of $\partial B(x,\eps)\setminus \eta[0,\sigma_{n-1}]$ containing $x+\eps$ and let $\sigma_n$ be the first time after $\tau_n$ that $\eta$ hits $(-\infty, y)$. Define $H_{2n-1}(\eps, x, y)$ to be the event that $\{\tau_n<T\}$.
Define $\hat{H}_{2n}(\eps, x, y)$ to be the event that $\{\sigma_n<T\}$.
In the definition of $H_{2n-1}(\eps, x, y)$ and $\hat{H}_{2n}(\eps, x, y)$,  we are particular interested in the case when $x$ is large. Roughly speaking, the event $H_{2n-1}(\eps, x, y)$ means that $\eta$ makes at least $(2n-1)$ crossings between $B(x,\eps)$ and $(-\infty, y)$. Imagine that $\eta$ is the interface in the discrete model, then $H_{2n-1}(\eps,x,y)$ interprets the event that there are $2n-1$ arms going from $B(x,\eps)$ to far away place. The event $\hat{H}_{2n}(\eps, x, y)$ means that $\eta$ makes at least $2n$ crossings between $B(x,\eps)$ and $(-\infty, y)$. Imagine that $\eta$ is the interface in the discrete model, then $\hat{H}_{2n}(\eps,x,y)$ interprets the event that there are $2n$ arms going from $B(x,\eps)$ to far away place. See Figure \ref{fig::boundaryarms}(a).

\begin{figure}[ht!]
\begin{subfigure}[b]{0.49\textwidth}
\begin{center}
\includegraphics[width=0.9\textwidth]{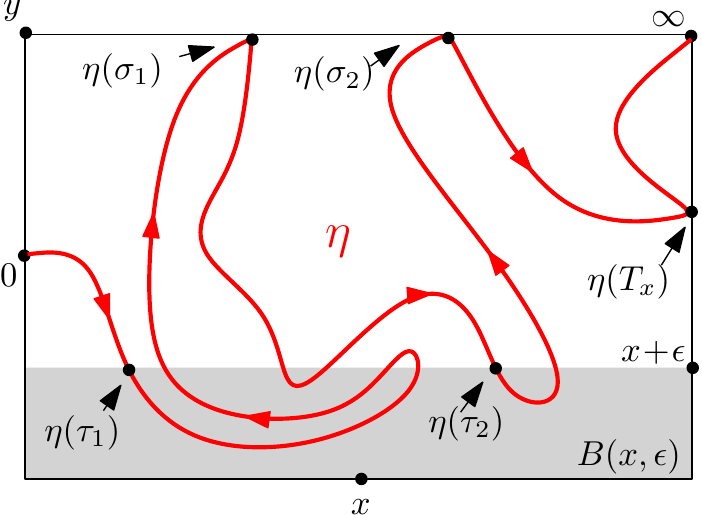}
\end{center}
\caption{This figure indicates $\hat{H}_4$. The stopping times\\ $\tau_1<\sigma_1<\tau_2<\sigma_2<T_x$ are indicated in the figure. }
\end{subfigure}
\begin{subfigure}[b]{0.49\textwidth}
\begin{center}\includegraphics[width=0.9\textwidth]{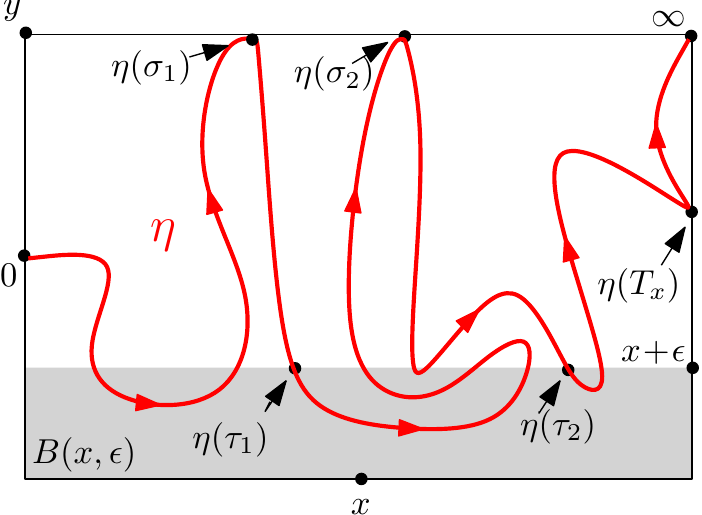}
\end{center}
\caption{This figure indicates $H_4$. The stopping times\\ $\sigma_1<\tau_1<\sigma_2<\tau_2<T_x$ are indicated in the figure. }
\end{subfigure}
\caption{\label{fig::boundaryarms} The explanation of the definition of the crossing events. The gray part is the ball $B(x,\eps)$.}
\end{figure}

Next, we define the crossing event $H_{2n}$ (resp. $\hat{H}_{2n+1}$) that $\eta$ crosses between the half-infinite line $(-\infty, y)$ and the ball $B(x,\eps)$ at least $2n$ times (resp. at least $2n+1$ times) for $n\ge 0$. Set $\tau_0=\sigma_0=0$. Let $\sigma_1$ be the first time that $\eta$ hits $(-\infty, y)$ and $\tau_1$ be the first time after $\sigma_1$ that $\eta$ hits the connected component of $\partial B(x,\eps)\setminus \eta[0,\sigma_1]$ containing $x+\eps$. For $n\ge 1$, let $\sigma_n$ be the first time after $\tau_{n-1}$ that $\eta$  hits $(-\infty, y)$ and $\tau_n$ be the first time after $\sigma_n$ that $\eta$ hits the connected component of $\partial B(x,\eps)\setminus \eta[0,\sigma_n]$ containing $x+\eps$.
Define $H_{2n}(\eps, x, y)$ to be the event that $\{\tau_{n}<T\}$.
Define $\hat{H}_{2n+1}(\eps, x, y)$ to be the event that $\{\sigma_{n+1}<T\}$.
In the definition of $H_{2n}(\eps, x, y)$ and $\hat{H}_{2n+1}(\eps, x, y)$. we are interested in the case when $x$ is of the same size as $\eps$ and $y$ is large. Roughly speaking, the event $H_{2n}(\eps, x,y)$ means that $\eta$ makes at least $2n$ crossings between $(-\infty, y)$ and $B(x,\eps)$. Imagine that $\eta$ is the interface in the discrete model, then $H_{2n}(\eps, x, y)$ interprets the event that there are $2n$ arms going from $B(x,\eps)$ to far away place.
The event $\hat{H}_{2n+1}(\eps, x,y)$ means that $\eta$ makes at least $2n+1$ crossings between $(-\infty, y)$ and $B(x,\eps)$. Imagine that $\eta$ is the interface in the discrete model, then $\hat{H}_{2n+1}(\eps, x, y)$ interprets the event that there are $2n+1$ arms going from $B(x,\eps)$ to far away place.
See Figure \ref{fig::boundaryarms}(b).

Note that in the definition of $H_{2n-1}$ and $\hat{H}_{2n}$, we start from $\tau_1$ and
\[H_{2n-1}(\eps, x, y)=\{\tau_1<\sigma_1<\tau_2<\cdots<\tau_n<T\},\quad
\hat{H}_{2n}(\eps, x, y)=\{\tau_1<\sigma_1<\tau_2<\cdots<\tau_n<\sigma_n<T\}.\]
In the definition of $H_{2n}$ and $\hat{H}_{2n+1}$, we start from $\sigma_1$ and
\[H_{2n}(\eps, x, y)=\{\sigma_1<\tau_1<\sigma_2<\cdots<\tau_n<T\},\quad \hat{H}_{2n+1}(\eps, x, y)=\{\sigma_1<\tau_1<\sigma_2<\cdots<\tau_n<\sigma_{n+1}<T\}.\] The two sequences of stopping times are defined in different ways. Readers may wander why we do not define the events using the same sequence of stopping times. We realize that the definition using the same sequence of stopping times causes ambiguity. Therefore, we decide to define these events in the above way. The advantages of the current definition will become clear in the proofs.

We define the arm exponents as follows. Set $\alpha_0^+=0$. For $n\ge 1$ and $\kappa\in (0,8)$, define
\begin{equation}\label{eqn::boundary_arm_formula_less8}
\alpha_{2n-1}^+=n(4n+4-\kappa)/\kappa,\quad \alpha^+_{2n}=n(4n+8-\kappa)/\kappa.
\end{equation}
For $n\ge 1$ and $\kappa\ge 8$, define
\begin{equation}\label{eqn::boundary_arm_formula_greater8}
\alpha_{2n-1}^+=(n-1)(4n+\kappa-8)/\kappa,\quad \alpha_{2n}^+=n(4n+\kappa-8)/\kappa.
\end{equation}

\begin{theorem}\label{thm::boundary_arm_greater4}
Fix $\kappa>4$. The crossing events $H_{2n-1}(\eps, x, y)$ and $H_{2n}(\eps, x, y)$ are defined as above.
Then, for any $y\le 0<\eps\le x$ and $n\ge 1$, we have
\begin{equation}\label{eqn::boundary_arm_greater4_odd}
\PP[H_{2n-1}(\eps, x, y)]\asymp \left(\frac{x}{x-y}\right)^{\alpha_{2n-2}^+}\left(\frac{\eps}{x}\right)^{\alpha_{2n-1}^+},
\end{equation}
\begin{equation}\label{eqn::boundary_arm_greater4_even}
\PP[H_{2n}(\eps, x, y)]\asymp \left(\frac{x}{x-y}\right)^{\alpha_{2n}^+}\left(\frac{\eps}{x}\right)^{\alpha_{2n-1}^+},
\end{equation}
where the constants in $\asymp$ depend only on $\kappa$ and $n$. In particular, fix some $\delta>0$, we have
\[\PP[H_{2n-1}(\eps, x, y)]\asymp \eps^{\alpha_{2n-1}^+},\quad\text{provided }\delta\le x\le 1/\delta, -1/\delta\le y\le 0,\]
\[\PP[H_{2n}(\eps, x, y)]\asymp\eps^{\alpha_{2n}^+},\quad \text{provided }\eps\le x\le \eps/\delta, -1/\delta\le y\le -\delta,\]
where the constants in $\asymp$ depend only on $\kappa, n$ and $\delta$.
\end{theorem}

By a similar proof, we could obtain a similar result as Theorem \ref{thm::boundary_arm_greater4} for $\SLE_{\kappa}(\rho)$ curve in the case that $x$ coincides with the force point. The exponents and a complete proof can be found in \cite[Section 3]{WuPolychromaticArmFKIsing}, where the conditions are loosen such that the force point may not be equal to $x$. One may also study the arm exponents for $\kappa\in(0,4]$. Whereas, when $\kappa\le 4$, the $\SLE$ curve does not touch the boundary, thus the above definition of the crossing events is not proper for $\kappa\le 4$. In Section \ref{sec::boundary_less4}, we have Theorem \ref{thm::boundary_arm_alternative} for the crossing events between a small circle and a half-infinite strip, where the arm exponents are defined in the same way as in (\ref{eqn::boundary_arm_formula_less8}). The proof of Theorem \ref{thm::boundary_arm_alternative} also works for $\SLE_{\kappa}(\rho)$ when $x$ coincides with the force point.

\begin{theorem}\label{thm::boundary_arm_hat}
Fix $\kappa\in (4,8)$. Set $\hat{\alpha}_0^+=0$. The crossing events $\hat{H}_{2n}(\eps, x, y)$ and $\hat{H}_{2n+1}(\eps, x, y)$ are defined as above.
For $n\ge 1$, define
\begin{equation}\label{eqn::boundary_arm_formula_hat}
\hat{\alpha}_{2n-1}^+=n(4n+\kappa-8)/\kappa,\quad \hat{\alpha}_{2n}^+=n(4n+\kappa-4)/\kappa.
\end{equation}
Then, for $y\le 0<\eps\le x$ and $n\ge 1$, we have
\begin{equation}\label{eqn::boundary_arm_hat_odd}
\PP\left[\hat{H}_{2n-1}(\eps, x, y)\right]\asymp \left(\frac{x}{x-y}\right)^{\hat{\alpha}_{2n-1}^+}\left(\frac{\eps}{x}\right)^{\hat{\alpha}_{2n-2}^+}, 
\end{equation}
\begin{equation}\label{eqn::boundary_arm_hat_even}
\PP\left[\hat{H}_{2n}(\eps, x, y)\right]\asymp \left(\frac{x}{x-y}\right)^{\hat{\alpha}_{2n-1}^+}\left(\frac{\eps}{x}\right)^{\hat{\alpha}_{2n}^+},
\end{equation}
where the constants in $\asymp$ depend only on $\kappa$ and $n$. In particular, fix some $\delta>0$, we have
\[\PP\left[\hat{H}_{2n-1}(\eps, x, y)\right]\asymp\eps^{\hat{\alpha}_{2n-1}^+},\quad \text{provided } \eps\le x\le \eps/\delta, -1/\delta\le y\le -\delta,\]
\[\PP\left[\hat{H}_{2n}(\eps, x, y)\right]\asymp\eps^{\hat{\alpha}_{2n}^+},\quad \text{provided }\delta\le x\le 1/\delta, -1/\delta\le y\le 0,\]
where the constants in $\asymp$ depend only on $\kappa, n$ and $\delta$.
\end{theorem}

It is worthwhile to spend some more words on the relation between $\alpha_n^+$ and $\hat{\alpha}_n^+$. In fact, we can also define the crossing events $\hat{H}_n(\eps, x, y)$ for $\kappa\in [0,4]$ and $\kappa\ge 8$. When $\kappa\le 4$, the $\SLE$ curve does not touch the boundary, thus the exponent $\hat{\alpha}_n^+$ coincides with $\alpha_{n-1}^+$. When $\kappa\ge 8$, the $\SLE$ curve is space-filling, thus the exponent $\hat{\alpha}_{n}^+$ coincides with $\alpha_{n+1}^+$. Whereas, when $\kappa\in (4,8)$, the exponent $\hat{\alpha}_n^+$ is distinct from $\alpha_n^+$ in general. In terms of discrete model, both $\alpha_n^+$ and $\hat{\alpha}_n^+$ interpret the boundary $n$-arm exponents, but their boundary conditions are different.

It is explained in \cite{SmirnovWernerCriticalExponents} that combining the following three facts would imply the arm exponents for the discrete model: (1) Identification between $\SLE_{\kappa}$ and the limit of the interface in critical lattice model; (2) The arm exponents of $\SLE_{\kappa}$;
(3) Crossing probabilities enjoy (approximate) multiplicativity property.  For critical Ising and FK-Ising model on $\Z^2$ with Dobrushin boundary conditions, the convergence to $\SLE_3$ and $\SLE_{16/3}$ respectively is derived in \cite{ChelkakSmirnovIsing, CDCHKSConvergenceIsingSLE}, and the multiplicativity is derived in \cite{ChelkakDuminilHonglerCrossingprobaFKIsing}. Therefore, we could derive the arm exponents for these two models. See more details in \cite{WuPolychromaticArmFKIsing, WuAlternatingArmIsing}. Moreover, the formula of $\alpha^+_{2n-1}$ in (\ref{eqn::boundary_arm_formula_less8}) was predicted by KPZ in \cite[Equations (11.42), (11.44)]{DuplantierFractalGeometry}.

\smallbreak
\noindent\textbf{Relation to previous results.} The formula of $\alpha_n^+$ and $\alpha_n$ for $\kappa=6$ was obtained in \cite{LawlerSchrammWernerExponent1, SmirnovWernerCriticalExponents}. The exponent $\alpha_1^+$ is related to the Hausdorff dimension of the intersection of $\SLE_{\kappa}$ with the real line which is $1-\alpha_1^+$ when $\kappa>4$. This dimension was obtained in \cite{AlbertsSheffieldDimension}. The most important ingredients in proving Theorem \ref{thm::boundary_arm_greater4} is the Laplace transform of the derivatives of the conformal map in $\SLE$ evolution, which was obtained in \cite{LawlerMinkowskiSLERealLine}.
\smallbreak
\noindent\textbf{Acknowledgment.} The authors acknowledge Hugo Duminil-Copin, Christophe Garban, Gregory Lawler,
Stanislav Smirnov, Vincent Tassion, Brent Werness, and David Wilson for helpful discussions. Hao Wu's work is supported by the NCCR/SwissMAP, the ERC AG COMPASP, the Swiss NSF.
Dapeng Zhan's work is partially supported by NSF DMS-1056840.

\section{Preliminaries}\label{sec::preliminaries}
\noindent\textbf{Notations.} We denote by $f\lesssim g$ if $f/g$ is bounded from above by universal finite constants, by $f\gtrsim g$ if $f/g$ is bounded from below by universal positive constants, and by $f\asymp g$ if $f\lesssim g$ and $f\gtrsim g$.

\noindent For $z\in\C, y\in\R, r>0$.
\begin{align*}
B(z, r)&=\{w\in\C: |w-z|<r\}, \quad \U=B(0,1);
\end{align*}
For two subsets $A, B\subset\C$,
\[\dist(A, B)=\inf\{|x-y|: x\in A, y\in B\}.\]
Let $\Omega$ be an open set and let $V_1, V_2$ be two sets such that $V_1\cap\overline{\Omega}\neq\emptyset$ and $V_2\cap\overline{\Omega}\neq\emptyset$. We denote the extremal distance between $V_1$ and $V_2$ in $\Omega$ by $d_{\Omega}(V_1, V_2)$, see \cite[Section 4]{AhlforsConformalInvariants} for the definition.

\subsection{$\HH$-hull and Loewner chain}
We call a compact subset $K$ of $\overline{\HH}$ an $\HH$-hull if $\HH\setminus K$ is simple connected. Riemann's Mapping Theorem asserts that there exists a unique conformal map $g_K$ from $\HH\setminus K$ onto $\HH$ such that
\[\lim_{|z|\to\infty}|g_K(z)-z|=0.\]
We call such $g_K$ the conformal map from $\HH\setminus K$ onto $\HH$ normalized at $\infty$.
The limit $\hcap(K):=\lim_{|z|\to\infty} z(g_K(z)-z)$ exists and is called the half-plane capacity of $K$.
\begin{lemma}\label{lem::extremallength_argument}
Fix $x>0$ and $\eps>0$.
Let $K$ be an $\HH$-hull and let $g_K$ be the conformal map from $\HH\setminus K$ onto $\HH$ normalized at $\infty$. Assume that
\[x>\max(K\cap\R).\]
Denote by $\gamma$ the connected component of $\HH\cap (\partial B(x,\eps)\setminus K)$ whose closure contains $x+\eps$. Then $g_K(\gamma)$ is contained in the ball with center $g_K(x+\eps)$ and radius $3(g_K(x+3\eps)-g_K(x+\eps))$. Hence $g_K(\gamma)$ is also contained in the ball with center $g_K(x+3\eps)$ and radius $8\eps g_K'(x+3\eps)$.
\end{lemma}
\begin{proof}
Define $r^*=\sup\{|z-g_K(x+\eps)|: z\in g_K(\gamma)\}$.
It is sufficient to show
\begin{equation}\label{eqn::extremal_aux}
r^*\le 3(g_K(x+3\eps)-g_K(x+\eps)).
\end{equation}
We will prove (\ref{eqn::extremal_aux}) by estimates on the extremal distance:
\[d_{\HH}(g_K(\gamma), [g_K(x+3\eps), \infty)).\]
By the conformal invariance and the comparison principle \cite[Section 4.3]{AhlforsConformalInvariants}, we can obtain the following lower bound.
\begin{align*}
d_{\HH}(g_K(\gamma), [g_K(x+3\eps), \infty))&=d_{\HH\setminus K}(\gamma, [x+3\eps, \infty))\\
&\ge d_{\HH\setminus B(x,\eps)}(B(x,\eps), [x+3\eps, \infty))\\
&=d_{\HH\setminus \U}(\U, [3,\infty))=d_{\HH}([-1,0], [1/3,\infty)).
\end{align*}
On the other hand, we will give an upper bound. Recall a fact for extremal distance: for $x<y$ and $r>0$, the extremal distance in $\HH$ between $[y,\infty)$ and a connected set $S\subset\overline{\HH}$ with $x\in\overline{S}\subset\overline{B(x,r)}$ is maximized when $S=[x-r,x]$, see \cite[Chapter I-E, Chapter III-A]{AhlforsQuasiconformal}. Since $g_K(\gamma)$ is connected and $g_K(x+\eps)\in \R\cap\overline{g_K(\gamma)}$, by the above fact, we have the following upper bound.
\begin{align*}
d_{\HH}(g_K(\gamma), [g_K(x+3\eps), \infty))&\le d_{\HH}([g_K(x+\eps)-r^*, g_K(x+\eps)], [g_K(x+3\eps), \infty))\\
&=d_{\HH}\left([-r^*,0], \left[g_K(x+3\eps)-g_K(x+\eps),\infty\right)\right).
\end{align*}
Combining the lower bound with the upper bound, we have
\[d_{\HH}([-1,0], [1/3,\infty))\le d_{\HH}\left([-r^*,0], \left[g_K(x+3\eps)-g_K(x+\eps),\infty\right)\right).\]
This implies (\ref{eqn::extremal_aux}) and completes the proof.
\end{proof}
\begin{lemma}\label{lem::image_insideball}
Fix $z\in\overline{\HH}$ and $\eps>0$. Let $K$ be an $\HH$-hull and let $g_K$ be the conformal map from $\HH\setminus K$ onto $\HH$ normalized at $\infty$. Assume that
\[\dist(K, z)\ge 16\eps.\]
Then $g_K(B(z,\eps))$ is contained in the ball with center $g_K(z)$ and radius $4\eps |g_K'(z)|$.
\end{lemma}
\begin{proof}
By Koebe 1/4 theorem, we know that
\[\dist(g_K(K), g_K(z))\ge d:=4\eps |g_K'(z)|. \]
Let $h=g_K^{-1}$ restricted to $B(g_K(z), d)$. Applying Koebe 1/4 theorem to $h$, we know that
\[\dist(z, \partial h(B(g_K(z), d)))\ge d|h'(g_K(z))|/4=\eps.\]
Therefore $h(B(g_K(z), d))$ contains the ball $B(z,\eps)$, and this implies that $B(g_K(z), d)$ contains the ball $g_K(B(z,\eps))$ as desired.
\end{proof}

\medbreak
Loewner chain is a collection of $\HH$-hulls $(K_{t}, t\ge 0)$ associated with the family of conformal maps $(g_{t}, t\ge 0)$ obtained by solving the Loewner equation: for each $z\in\mathbb{H}$,
\begin{equation}\label{loewner}
\partial_{t}{g}_{t}(z)=\frac{2}{g_{t}(z)-W_{t}}, \quad g_{0}(z)=z,
\end{equation}
where $(W_t, t\ge 0)$ is a one-dimensional continuous function which we call the driving function. Let $T_z$ be the swallowing time of $z$ defined as $\sup\{t\ge 0: \min_{s\in[0,t]}|g_{s}(z)-W_{s}|>0\}$.
Let $K_{t}:=\overline{\{z\in\mathbb{H}: T_{z}\le t\}}$. Then $g_{t}$ is the unique conformal map from $H_{t}:=\mathbb{H}\backslash K_{t}$ onto $\mathbb{H}$ normalized at $\infty$.

Here we spend some words about the evolution of a point $y\in\R$ under $g_t$. We assume $y\le 0$, the case of $y\ge 0$ can be analyzed similarly. There are two possibilities: if $y$ is not swallowed by $K_t$, then we define $Y_t=g_t(y)$; if $y$ is swallowed by $K_t$, then we define $Y_t$ to the be image of the leftmost of point of $K_t\cap\R$ under $g_t$. The process $Y_t$ is decreasing in $t$, and it is uniquely characterized by the following equation:
\[Y_t=y+\int_0^t \frac{2ds}{Y_s-W_s},\quad Y_t\le W_t,\quad \forall t\ge 0.\]
In this paper, we may write $g_t(y)$ for the process $Y_t$.
Consider two points $x\ge 0\ge  y$ in $\R$. By the above fact, we have
\[g_t(x)=x+\int_0^t\frac{2ds}{g_s(x)-W_s},\quad g_t(y)=y+\int_0^t\frac{2ds}{g_s(y)-W_s}, \quad g_t(y)\le W_t\le g_t(x).\]
Therefore, the quantity $g_t(x)-g_t(y)$ is increasing in $t$. We will use this fact in the paper without reference.

\subsection{SLE processes}
An $\SLE_{\kappa}$ is the random Loewner chain $(K_{t}, t\ge 0)$ driven by $W_t=\sqrt{\kappa}B_t$ where $(B_t, t\ge 0)$ is a standard one-dimensional Brownian motion.
In \cite{RohdeSchrammSLEBasicProperty}, the authors prove that $(K_{t}, t\ge 0)$ is almost surely generated by a continuous transient curve, i.e. there almost surely exists a continuous curve $\eta$ such that for each $t\ge 0$, $H_{t}$ is the unbounded connected component of $\mathbb{H}\backslash\eta[0,t]$ and that $\lim_{t\to\infty}|\eta(t)|=\infty$.

We can define an SLE$_{\kappa}(\rho^{L}; \rho^{R})$ process with two force points $(x^{L}; x^{R})$ where $x^{L}\le 0\le x^{R}$. It is the Loewner chain driven by $W_{t}$ which is the solution to the following systems of SDEs:
\[dW_{t}=\sqrt{\kappa}dB_{t}+\frac{\rho^{L} dt}{W_{t}-V_{t}^{L}}+\frac{\rho^{R} dt}{W_{t}-V_{t}^{R}}, \quad W_{0}=0;\]
\[dV^{L}_{t}=\frac{2dt}{V^{L}_{t}-W_{t}}, \quad V^{L}_{0}=x^{L}; \quad dV^{R}_{t}=\frac{2dt}{V^{R}_{t}-W_{t}}, \quad V^{R}_{0}=x^{R}.\]

The solution exists up to the first time that $W$ hits $V^L$ or $V^R$.
When $\rho^L>-2$ and $\rho^R>-2$, the solution exists for all times $t\ge 0$, and the corresponding Loewner chain is almost surely generated by a continuous curve which is almost surely transient (\cite[Section 2]{MillerSheffieldIG1}). There are two special values of $\rho$: $\kappa/2-2$ and $\kappa/2-4$. When $\rho^R\ge \kappa/2-2$, then the curve will never hits $[x^R,\infty)$. When $\rho^R\le \kappa/2-4$, then the curve will almost surely accumulates at $x^R$ at finite time. See \cite[Lemma 15]{DubedatSLEDuality}.

From Girsanov Theorem, it follows that the law of an $\SLE_{\kappa}(\rho^L;\rho^R)$ process can be constructed by reweighting the law of an ordinary $\SLE_{\kappa}$.
\begin{lemma}\label{lem::sle_mart}
Suppose $x^L<0<x^R$, define
\begin{align*}
M_t=&g_t'(x^L)^{\rho^L(\rho^L+4-\kappa)/(4\kappa)}(W_t-g_t(x^L))^{\rho^L/\kappa}\times g_t'(x^R)^{\rho^R(\rho^R+4-\kappa)/(4\kappa)}(g_t(x^R)-W_t)^{\rho^R/\kappa}\\
&\times (g_t(x^R)-g_t(x^L))^{\rho^L\rho^R/(2\kappa)}.
\end{align*}
Then $M$ is a local martingale for $\SLE_{\kappa}$ and the law of $\SLE_{\kappa}$ weighted by $M$ (up to the first time that $W$ hits one of the force points) is equal to the law of $\SLE_{\kappa}(\rho^L;\rho^R)$ with force points $(x^L; x^R)$.
\end{lemma}
\begin{proof}
\cite[Theorem 6]{SchrammWilsonSLECoordinatechanges}.
\end{proof}

\begin{lemma}\label{lem::stochasticbound1}
Fix $\kappa>0$ and $\nu\le\kappa/2-4$. Suppose $y\le 0<x$. Let $\eta$ be an $\SLE_{\kappa}(\nu)$ in $\HH$ from 0 to $\infty$ with force point $x$. Since $\nu\le\kappa/2-4$, the curve $\eta$ accumulates at the point $x$ at almost surely finite time which is denoted by $T$. Then we have, for $\lambda\le 0$,
\[\E\left[\left(g_T(x)-g_T(y)\right)^{\lambda}\right]\asymp (x-y)^{\lambda}, \]
where the constants in $\asymp$ depend only $\kappa, \nu$ and $\lambda$.
\end{lemma}
\begin{proof}
Since the quantity $g_t(x)-g_t(y)$ is increasing in $t$, we have $g_{T}(x)-g_{T}(y)\ge (x-y)$. This implies the upper bound. We only need to show the lower bound. To this end, we will compare $\eta$ with $\SLE_{\kappa}(\nu)$ with force point $x-y$ and show that the law of $(g_T(x)-g_T(y))/(x-y)$ is stochastically dominated by a random variable whose law depends only $\kappa, \nu$. By the scaling invariance of $\SLE_{\kappa}(\nu)$, we may assume $x-y=1$.

Let $\tilde{\eta}$ be an $\SLE_{\kappa}(\nu)$ with force point $1$, and define $\tilde{W}, \tilde{g}_t, \tilde{T}$ accordingly. Define $\tilde{V}_t$ to be the image of the leftmost point of $\tilde{\eta}[0,t]\cap\R$ under $\tilde{g}_t$. Set
\[\tilde{J}_t=\frac{\tilde{W}_t-\tilde{V}_t}{\tilde{g}_t(1)-\tilde{V}_t}.\]
Define the stopping time $\tau=\inf\{t: \tilde{J}_t=-y\}$.
Note that $\tilde{J}_0=0, \tilde{J}_{\tilde{T}}=1$ and $\tilde{J}$ is continuous, we have that $0\le \tau\le \tilde{T}$. Given $\tilde{\eta}[0,\tau]$, the process $(\tilde{\eta}(t+\tau), 0\le t\le \tilde{T}-\tau)$,
under the map
\[f(z)=\frac{\tilde{g}_{\tau}(z)-\tilde{W}_{\tau}}{\tilde{g}_{\tau}(1)-\tilde{V}_{\tau}},\]
has the same law as $(\eta(t), 0\le t\le T)$ after a linear time-change. Therefore, given $\tilde{\eta}[0,\tau]$, we have
\[\frac{\tilde{g}_{\tilde{T}}(1)-\tilde{V}_{\tilde{T}}}{\tilde{g}_{\tau}(1)-\tilde{V}_{\tau}}\overset{d}{=}g_T(x)-g_T(y).\]
Since $\tilde{g}_{\tau}(1)-\tilde{V}_{\tau}\ge 1$, we may conclude that
the quantity $(g_T(x)-g_T(y))$ is stochastically dominated from above by $(\tilde{g}_{\tilde{T}}(1)-\tilde{V}_{\tilde{T}})$.
To complete the proof, it is sufficient to show
\begin{equation}\label{eqn::stochasticbound1_aux}
\tilde{\E}\left[\left(\tilde{g}_{\tilde{T}}(1)-\tilde{V}_{\tilde{T}}\right)^{\lambda}\right]\gtrsim 1,
\end{equation}
where $\tilde{\PP}$ denotes the law of $\SLE_{\kappa}(\nu)$ with force point $1$. Define the event
\[\tilde{F}=\{\tilde{g}_{\tilde{T}}(1)-\tilde{V}_{\tilde{T}}\le 4\}.\]
It is clear that $\tilde{\PP}[\tilde{F}]$ is strictly positive and depends only on $\kappa$ and $\nu$, thus
\[\tilde{\E}\left[\left(\tilde{g}_{\tilde{T}}(1)-\tilde{V}_{\tilde{T}}\right)^{\lambda}\right]\ge 4^{\lambda}\tilde{\PP}[\tilde{F}].\]
This implies (\ref{eqn::stochasticbound1_aux}) and completes the proof.
\end{proof}

\begin{lemma}\label{lem::stochasticbound2}
Fix $\kappa>4$ and $\nu\ge\kappa/2-2$. Suppose $y<0<x$, let $\eta$ be an $\SLE_{\kappa}(\nu)$ with force point $x$. For $c>0$ small, define
\[\sigma=\inf\{t: \eta(t)\in (-\infty, y]\},\quad F=\{\dist(\eta[0,\sigma], x)\ge cx\}.\] Then there exists a constant $c\in (0,1)$ depending only on $\kappa$ and $\nu$ such that, for $\lambda\le 0$,
\[\E\left[\left(g_{\sigma}(x)-g_{\sigma}(y)\right)^{\lambda}1_{F}\right]\asymp (x-y)^{\lambda},\]
where the constants in $\asymp$ depend only on $\kappa, \nu$ and $\lambda$.
\end{lemma}
\begin{proof}
Since the quantity $g_t(x)-g_t(y)$ is increasing in $t$, we have $g_{\sigma}(x)-g_{\sigma}(y)\ge (x-y)$. This implies the upper bound. We only need to show the lower bound. We may assume that $x-y=1$.
We first argue that
\begin{equation}\label{eqn::stochasticbound2_aux}
\E\left[\left(g_{\sigma}(x)-g_{\sigma}(y)\right)^{\lambda}\right]\asymp (x-y)^{\lambda}.
\end{equation}
The proof of (\ref{eqn::stochasticbound2_aux}) is similar to the proof of Lemma \ref{lem::stochasticbound1}. Let $\tilde{\eta}$ be an $\SLE_{\kappa}(\nu)$ with force point $0^+$. Define $\tilde{W}, \tilde{g}$ accordingly and let $\tilde{\sigma}$ be the first time that $\tilde{\eta}$ hits $(-\infty, -1)$. Let $\tilde{V}_t$ be the evolution of the force point. Define
\[\tilde{J}_t=\frac{\tilde{V}_t-\tilde{W}_t}{\tilde{V}_t-\tilde{g}_t(-1)},\quad \tau:=\inf\{t: \tilde{J}_t=x\}.\]
Given $\tilde{\eta}[0,\tau]$, the process $(\tilde{\eta}(t+\tau), 0\le t\le \tilde{\sigma}-\tilde{\tau})$ under the map
\[f(z)=\frac{\tilde{g}_{\tau}(z)-\tilde{W}_{\tau}}{\tilde{V}_\tau-\tilde{g}_{\tau}(-1)}\]
has the same law as $(\eta(t), 0\le t\le \sigma)$ after a linear time change. In particular,
\[\frac{\tilde{V}_{\tilde{\sigma}}-\tilde{g}_{\tilde{\sigma}}(-1)}{\tilde{V}_\tau-\tilde{g}_{\tau}(-1)}\overset{d}{=}g_{\sigma}(x)-g_{\sigma}(y).\]
Since $\tilde{V}_\tau-\tilde{g}_{\tau}(-1)\ge 1$, we know that $(g_{\sigma}(x)-g_{\sigma}(y))$ is stochastically dominated from above by $(\tilde{V}_{\tilde{\sigma}}-\tilde{g}_{\tilde{\sigma}}(-1))$, thus
\[\E\left[\left(g_{\sigma}(x)-g_{\sigma}(y)\right)^{\lambda}\right]\ge \tilde{\E}\left[\left(\tilde{V}_{\tilde{\sigma}}-\tilde{g}_{\tilde{\sigma}}(-1)\right)^{\lambda}\right]\asymp 1.\]
This implies (\ref{eqn::stochasticbound2_aux}). Next, we prove the conclusion. By the scaling invariance of $\SLE_{\kappa}(\nu)$ process we know that the probability $\PP[\dist(\eta, x)< cx]$ only depends on $c$. We denote this probability by $p(c)$. Since $\nu\ge \kappa/2-2$, we know that $p(c)\to 0$ as $c\to 0$. Therefore, by (\ref{eqn::stochasticbound2_aux}), we have
\[1\asymp \E\left[\left(g_{\sigma}(x)-g_{\sigma}(y)\right)^{\lambda}\right]\le \E\left[\left(g_{\sigma}(x)-g_{\sigma}(y)\right)^{\lambda}1_{F}\right]+p(c).\]
This implies the conclusion.
\end{proof}

\section{Boundary Arm Exponents for $\kappa>4$}\label{sec::boundary_greater4}
\subsection{Estimate on the derivative}
\begin{proposition}\label{prop::sle_boundary_estimate}
Fix $\kappa>0$ and let $\eta$ be an $\SLE_{\kappa}$ in $\HH$ from 0 to $\infty$. Let $O_t$ be the image of the rightmost point of $K_t\cap \R$ under $g_t$. Set $\Upsilon_t=(g_1(1)-O_t)/g_t'(1)$.
For $\eps\in (0,1)$, define
\[\hat{\tau}_{\eps}=\inf\{t: \Upsilon_t=\eps\},\quad T_0=\inf\{t: \eta(t)\in [1, \infty)\}.\]
For $\lambda\ge 0$, define
\[u_1(\lambda)=\frac{1}{\kappa}(4-\kappa/2)+\frac{1}{\kappa}\sqrt{4\kappa\lambda+(4-\kappa/2)^2}.\]
For $b\in\R$, assume that
\begin{equation}\label{eqn::requirement_b}
\kappa\lambda-\kappa u_1(\lambda)+8-2\kappa<\kappa b\le \kappa\lambda+\kappa u_1(\lambda).
\end{equation}
Then we have
\begin{equation}\label{eqn::estimate_derivative}
\E\left[(g_{\hat{\tau}_{\eps}}(1)-W_{\hat{\tau}_{\eps}})^{\lambda-b}g_{\hat{\tau}_{\eps}}'(1)^{b} 1_{\{\hat{\tau}_{\eps}<T_0\}}\right]\asymp\eps^{u_1(\lambda)+\lambda-b},
\end{equation}
where the constants in $\asymp$ depend only on $\kappa$ and $\lambda, b$.
\end{proposition}
Attention that, in Proposition \ref{prop::sle_boundary_estimate}, we use the stopping time $\hat{\tau}_{\eps}$ instead of $\tau_{\eps}$ which is defined to be the first time that $\eta$ hits $B(1,\eps)$. Due to Koebe 1/4 thoerem, these two times are very close: 
\[\tau_{4\eps}\le \hat{\tau}_{\eps}\le\tau_{\eps/4}.\]
Due to technical reason, we only prove the conclusion in Proposition \ref{prop::sle_boundary_estimate} for the time $\hat{\tau}_{\eps}$, but this is sufficient for our purpose later in the paper. 
\begin{lemma}\label{lem::sle_rho_integrable}
Fix $\kappa >0$ and $\nu\le \kappa/2-4$. Let $\eta$ be an $\SLE_{\kappa}(\nu)$ in $\HH$ from 0 to $\infty$ with force point $1$. Denote by $W$ the driving function, $V$ the evolution of the force point.  Let $O_t$ be the image of the rightmost point of $K_t\cap \R$ under $g_t$. Set $\Upsilon_t=(g_t(1)-O_t)/g_t'(1)$ and $\sigma(s)=\inf\{t: \Upsilon_t=e^{-2s}\}$.
Set $J_t=(V_t-O_t)/(V_t-W_t)$. Let $T_0=\inf\{t: \eta(t)\in [1,\infty)\}$. We have, for $\beta> 0$,
\begin{equation}\label{eqn::integrable_J_sigma}
\E\left[J_{\sigma(s)}^{-\beta}1_{\{\sigma(s)<T_0\}}\right]\asymp 1,\quad \text{when }8+2\nu+\kappa\beta<2\kappa,
\end{equation}
where the constants in $\asymp$ depend only on $\kappa, \nu, \beta$.
\end{lemma}
\begin{proof}
Since $0\le J_t\le 1$, we only need to show the upper bounds. 
Define $X_t=V_t-W_t$.
We know that
\[dW_t=\sqrt{\kappa}dB_t+\frac{\nu dt}{W_t-V_t},\quad dV_t=\frac{2dt}{V_t-W_t},\]
where $B$ is a standard 1-dimensional Brownian motion.
By It\^o's formula, we have that
\[dJ_t=\frac{J_t}{X_t^2}\left(\kappa-\nu-2-\frac{2}{1-J_t}\right)dt+\frac{J_t}{X_t}\sqrt{\kappa}dB_t,\quad d\Upsilon_t=\Upsilon_t\frac{-2J_t dt}{X_t^2(1-J_t)}.\]
Recall that $\sigma(s)=\inf\{t: \Upsilon_t=e^{-2s}\}$, and denote by $\hat{X}, \hat{J}, \hat{\Upsilon}$ the processes indexed by $\sigma(s)$. Then we have that
\[d\sigma(s)=\hat{X}_s^2\frac{1-\hat{J}_s}{\hat{J}_s}ds,\quad d\hat{J}_s=\left(\kappa-\nu-4-(\kappa-\nu-2)\hat{J}_s\right)ds+\sqrt{\kappa \hat{J}_s(1-\hat{J}_s)}d\hat{B}_s,\]
where $\hat{B}$ is a standard 1-dimensional Brownian motion.
By \cite[Equations (56), (62)]{LawlerMinkowskiSLERealLine} and \cite[Appendix B]{ZhanErgodicityTipSLE}, we know that $\hat{J}$ has an invariant density on $(0,1)$, which is proportional to $y^{1-(8+2\nu)/\kappa}(1-y)^{4/\kappa-1}$.
Moreover, since $\hat J_0=1$, by a standard coupling argument, we may couple $(\hat J_s)$ with the stationary process $(\tilde J_s)$ that satisfies the same equation as $(\hat J_s)$, such that $\hat J_s\ge \tilde J_s$ for all $s\ge 0$. Then we get $\E[\hat{J}_s^{-\beta}]\le \E[\tilde J_s^{-\beta}]$, which is a finite constant if $8+2\nu+\kappa\beta<2\kappa$. This gives the upper bound in (\ref{eqn::integrable_J_sigma}) and completes the proof of (\ref{eqn::integrable_J_sigma}). 
\end{proof}

\begin{proof}[Proof of Proposition \ref{prop::sle_boundary_estimate}]
Let $O_t$ be the image of the rightmost point of $\eta[0,t]\cap\R$ under $g_t$.
Define
\[\Upsilon_t=\frac{g_t(1)-O_t}{g_t'(1)},\quad J_t=\frac{g_t(1)-O_t}{g_t(1)-W_t}.\]
Set
\[M_t=g_t'(1)^{\nu(\nu+4-\kappa)/(4\kappa)}(g_t(1)-W_t)^{\nu/\kappa}, \quad \text{where }\nu=-\kappa u_1(\lambda).\]
Then $M$ is a local martingale for $\eta$, and from Lemma \ref{lem::sle_mart}, the law of $\eta$ weighted by $M$ is the law of $\SLE_{\kappa}(\nu)$ with force point $1$.
Set $\beta=u_1(\lambda)+\lambda-b$.
 Then we have
\[M_t=(g_t(1)-W_t)^{\lambda-b}g_t'(1)^{b}\Upsilon_t^{-\beta}J_t^{\beta}.\]
At time $t=\hat{\tau}_{\eps}<\infty$, we have $\Upsilon_t=\eps$, thus
\[\E\left[(g_{\hat{\tau}_{\eps}}(1)-W_{\hat{\tau}_{\eps}})^{\lambda-b}g_{\hat{\tau}_{\eps}}'(1)^{b} 1_{\{\hat{\tau}_{\eps}<T_0\}}\right]\asymp \eps^\beta\E^*\left[\left(J^*_{\hat{\tau}_{\eps}^*}\right)^{-\beta}1_{\{\hat{\tau}_{\eps}^*<T_0^*\}}\right]\asymp \eps^{\beta},\]
where $\PP^*$ is the law of $\SLE_{\kappa}(\nu)$ with force point $x$ and $\eta^*, J^*, \hat{\tau}_{\eps}^*, T_0^*$ are defined accordingly, and the last relation is due to (\ref{eqn::integrable_J_sigma}). 
\end{proof}

\begin{remark}\label{rem::sle_boundary_firststep}
Fix $\kappa>0$ and let $\eta$ be an $\SLE_{\kappa}$. For $x>\eps>0$,
let $u_1(\lambda)$ and $b$ be as in Proposition \ref{prop::sle_boundary_estimate}. By the scaling invariance of $\SLE$, we have
\begin{equation}\label{eqn::estimate_derivative_scale}
\E\left[(g_{\hat{\tau}_{\eps}}(x)-W_{\hat{\tau}_{\eps}})^{\lambda-b}g_{\hat{\tau}_{\eps}}'(x)^{b} 1_{\{\hat{\tau}_{\eps}<T_0\}}\right]\asymp x^{-u_1(\lambda)}\eps^{u_1(\lambda)+\lambda-b},
\end{equation}
where the constants in $\asymp$ depend only on $\kappa$, and $\lambda, b$.
Taking $\lambda=b=0$, we have
\[\PP[\tau_{\eps}<\infty]\asymp \PP[\hat{\tau}_{\eps}<\infty]\asymp \left(\frac{\eps}{x}\right)^{\alpha_1^+},\quad \text{where }\alpha_1^+=u_1(0)=0\vee (8/\kappa-1).\]
This implies that (\ref{eqn::boundary_arm_greater4_odd}) holds for $n=1$.
\end{remark}

\subsection{From $2n-1$ to $2n$}\label{subsec::greater4_odd_even}
\begin{lemma}\label{lem::estimate_derivative_even}
Fix $\kappa>4$ and let $\eta$ be an $\SLE_{\kappa}$. For $y<0<x$, define
\[\sigma=\inf\{t: \eta(t)\in (-\infty, y]\},\quad T=\inf\{t: \eta(t)\in [x,\infty)\},\quad F=\{\dist(\eta[0,\sigma], x)\ge cx\},\]
where $c$ is the constant decided in Lemma \ref{lem::stochasticbound2}.
For $\lambda\ge 0$, define
\[u_2(\lambda)=\frac{1}{\kappa}(\kappa/2-2)+\frac{1}{\kappa}\sqrt{4\kappa\lambda+(\kappa/2-2)^2}.\]
Then we have, for $\lambda\ge 0$ and $b\le u_2(\lambda)$,
\begin{align*}
\E\left[g_{\sigma}'(x)^{\lambda}
(g_{\sigma}(x)-W_{\sigma})^b 1_{\{\sigma<T\}\cap F}\right]&\gtrsim x^{u_2(\lambda)}(x-y)^{b-u_2(\lambda)},\\
\E\left[g_{\sigma}'(x)^{\lambda}
(g_{\sigma}(x)-W_{\sigma})^b 1_{\{\sigma<T\}}\right]&\lesssim x^{u_2(\lambda)}(x-y)^{b-u_2(\lambda)},
\end{align*}
where the constants in $\gtrsim$ and $\lesssim$ depend only on $\kappa$ and $\lambda, b$.
\end{lemma}
\begin{proof}
Define
\[M_t=g_t'(x)^{\nu(\nu+4-\kappa)/(4\kappa)}(g_t(x)-W_t)^{\nu/\kappa},\quad \text{where }\nu=\kappa u_2(\lambda).\]
Then $M$ is a local martingale for $\eta$ and the law of $\eta$ weighted by $M$ is the law of $\SLE_{\kappa}(\nu)$ with force point $x$.
By the definition of $u_2$, we can also write
\[M_t=g_t'(x)^{\lambda}(g_t(x)-W_t)^{u_2(\lambda)}.\]
Thus
\[\E\left[g_{\sigma}'(x)^{\lambda}
(g_{\sigma}(x)-W_{\sigma})^b 1_{\{\sigma<T\}}\right]=M_0\E^*\left[\left(g^*_{\sigma^*}(x)-g^*_{\sigma^*}(y)\right)^{b-u_2(\lambda)}1_{\{\sigma^*<T^*\}}\right],\]
where $\PP^*$ denotes the law of $\SLE_{\kappa}(\nu)$ with force point $x$ and $\eta^*, g^*, \sigma^*$ and $T^*$ are defined accordingly. Since $\nu\ge\kappa/2-2$, the curve will never swallows $x$, thus $T^*=\infty$. Note that $M_0=x^{u_2(\lambda)}$. Therefore, proving the conclusion boils down to showing
\begin{align}
\E^*\left[\left(g^*_{\sigma^*}(x)-g^*_{\sigma^*}(y)\right)^{b-u_2(\lambda)}1_{F^*}\right]&\gtrsim (x-y)^{b-u_2(\lambda)},\quad \text{where }F^*=\{\dist(\eta^*[0,\sigma^*], x)\ge cx\};\label{eqn::estimate_derivative_even_lower}\\
\E^*\left[\left(g^*_{\sigma^*}(x)-g^*_{\sigma^*}(y)\right)^{b-u_2(\lambda)}\right]&\lesssim (x-y)^{b-u_2(\lambda)}.\label{eqn::estimate_derivative_even_upper}
\end{align}
Equation (\ref{eqn::estimate_derivative_even_lower}) is true by Lemma \ref{lem::stochasticbound2}. Since the quantity $(g^*_t(x)-g^*_t(y))$ is increasing in $t$, we have
\[(g^*_{\sigma^*}(x)-g^*_{\sigma^*}(y))\ge x-y.\]
Combining with the fact that $b-u_2(\lambda)\le 0$, we obtain (\ref{eqn::estimate_derivative_even_upper}).
\end{proof}

\begin{remark}\label{rem::hat_initial}
Taking $\lambda=b=0$ in Lemma \ref{lem::estimate_derivative_even}, we have
\[\PP[\sigma<T]\asymp x^{u_2(0)}.\]
This implies that (\ref{eqn::boundary_arm_hat_odd}) holds for $n=1$ with
\[\hat{\alpha}^+_1=u_2(0)=1-4/\kappa.\]
\end{remark}

\begin{lemma}\label{lem::greater4_odd_even}
Assume the same notations as in Theorem \ref{thm::boundary_arm_greater4}. Suppose that (\ref{eqn::boundary_arm_greater4_odd}) holds for $2n-1$, then (\ref{eqn::boundary_arm_greater4_even}) holds for $2n$.
\end{lemma}

\begin{proof}[Proof of Lemma \ref{lem::greater4_odd_even}, Upper Bound.]
Let $\eta$ be an $\SLE_{\kappa}$ and define
\[\sigma=\inf\{t: \eta(t)\in (-\infty, y]\},\quad T=\inf\{t: \eta(t)\in [x,\infty)\}.\]
We stop the curve at time $\sigma$. Let $\tilde{\eta}$ be the image of $\eta[\sigma,\infty)$ under the centered comformal map $f:=g_{\sigma}-W_{\sigma}$. Then $\tilde{\eta}$ is an $\SLE_{\kappa}$. Define $\tilde{H}_{2n-1}$ for $\tilde{\eta}$.

Given $\eta[0,\sigma]$ with $\sigma<T$, consider the event $H_{2n}(\eps, x, y)$. Denote by $\gamma$ the connected component of $B(x,\eps)\setminus\eta[0,\sigma]$ whose boundary contains $x+\eps$.
We wish to control the image of $(-\infty, y]$ and the image of $\gamma$ under $f$. We have the following observations.
\begin{itemize}
\item At time $\sigma$, we have $W_{\sigma}=g_{\sigma}(y)$, thus $f(y)=0$.
\item By Lemma \ref{lem::extremallength_argument}, we know that $f(\gamma)$ is contained in the ball with center $f(x+3\eps)$ and radius $8\eps f'(x+3\eps)$. 
\end{itemize}
Combining these two facts, we know that, given $\eta[0,\sigma]$ with $\sigma<T$, the event $H_{2n}(\eps, x, y)$ implies the event $\tilde{H}_{2n-1}(8\eps f'(x+3\eps), f(x+3\eps), 0)$. 
If $f(x+3\eps)\ge 8\eps f'(x+3\eps)$, 
by the assumption hypothesis, we have
\[\PP[H_{2n}(\eps, x, y)\cond \eta[0,\sigma], \sigma<T]\lesssim \left(\frac{\eps g_{\sigma}'(x+3\eps)}{g_{\sigma}(x+3\eps)-W_{\sigma}}\right)^{\alpha_{2n-1}^+}.\]
If $f(x+3\eps)\le 8\eps f'(x+3\eps)$, the above upper bound is trivially true. Therefore, the above upper bound always holds. Then
\[\PP[H_{2n}(\eps, x, y)]\lesssim \eps^{\alpha_{2n-1}^+}\E\left[g'_{\sigma}(x+3\eps)^{\alpha_{2n-1}^+}(g_{\sigma}(x+3\eps)-W_{\sigma})^{-\alpha_{2n-1}^+}1_{\{\sigma<T\}}\right].\]
To apply Lemma \ref{lem::estimate_derivative_even}, we only need to note that $T$ is the first time that $\eta$ swallows $x$ which happens before the first time that $\eta$ swallows $x+3\eps$. Note further that
\begin{equation}\label{eqn::from2n-1to2n}
u_2(\alpha_{2n-1}^+)=\alpha_{2n}^+-\alpha_{2n-1}^+.
\end{equation}
Thus, by Lemma \ref{lem::estimate_derivative_even}, we have
\[\PP[H_{2n}(\eps, x, y)]\lesssim \eps^{\alpha_{2n-1}^+}x^{\alpha_{2n}^+-\alpha_{2n-1}^+}(x-y)^{-\alpha_{2n}^+}=\left(\frac{x}{x-y}\right)^{\alpha_{2n}^+}\left(\frac{\eps}{x}\right)^{\alpha_{2n-1}^+}.\]
This completes the proof of the upper bound.
\end{proof}

\begin{proof}[Proof of Lemma \ref{lem::greater4_odd_even}, Lower Bound.]
Let $\eta$ be an $\SLE_{\kappa}$ and assume the same notations as in the proof of the upper bound. Define $F=\{\dist(\eta[0,\sigma], x)\ge c\eps\}$,
where $c$ is the constant decided in Lemma \ref{lem::stochasticbound2}.
We stop the curve at time $\sigma$. Let $\tilde{\eta}$ be the image of $\eta[\sigma,\infty)$ under the centered comformal map $f:=g_{\sigma}-W_{\sigma}$. Then $\tilde{\eta}$ is an $\SLE_{\kappa}$. Define $\tilde{H}_{2n-1}$ for $\tilde{\eta}$.

Given $\eta[0,\sigma]$ with $\{\sigma<T\}\cap F$, consider the event $H_{2n}(\eps, x, y)$. We wish to control the image of $(-\infty, y]$ and the image of $\partial B(x,\eps)$ under $f$. We have the following observations.
\begin{itemize}
\item At time $\sigma$, we have $W_{\sigma}=g_{\sigma}(y)$, thus $f(y)=0$.
\item On the event $F$, by Koebe 1/4 Theorem, we know that $f(B(x,\eps))$ contains the ball with center $f(x)$ and radius $cf'(x)\eps/4$. 
\end{itemize}
Combining these two facts, we know that, given $\eta[0,\sigma]$ with $\{\sigma<T\}\cap F$, the event $H_{2n}(\eps, x, y)$ contains the event $\tilde{H}_{2n-1}(f'(x)c\eps/4, f(x), 0)$. By the assumption hypothesis, we have
\[\PP[H_{2n}(\eps, x, y)\cond \eta[0,\sigma], \{\sigma<T\}\cap F]\gtrsim \left(\frac{\eps g_{\sigma}'(x)}{g_{\sigma}(x)-W_{\sigma}}\right)^{\alpha_{2n-1}^+}.\]
Therefore,
\[\PP[H_{2n}(\eps, x, y)]\gtrsim \eps^{\alpha_{2n-1}^+}\E\left[g'_{\sigma}(x)^{\alpha_{2n-1}^+}(g_{\sigma}(x)-W_{\sigma})^{-\alpha_{2n-1}^+}1_{\{\sigma<T\}\cap F}\right].\]
To apply Lemma \ref{lem::estimate_derivative_even}, we only need to note that $x\ge \eps$ and the event $F$ contains the event $\{\dist(\eta[0,\sigma], x)\ge cx\}$. By (\ref{eqn::from2n-1to2n}) and Lemma \ref{lem::estimate_derivative_even}, we have
\[\PP[H_{2n}(\eps, x, y)]\gtrsim \eps^{\alpha_{2n-1}^+}x^{\alpha_{2n}^+-\alpha_{2n-1}^+}(x-y)^{-\alpha_{2n}^+}=\left(\frac{x}{x-y}\right)^{\alpha_{2n}^+}\left(\frac{\eps}{x}\right)^{\alpha_{2n-1}^+}.\]
This completes the proof of the lower bound.
\end{proof}

\subsection{From $2n$ to $2n+1$}\label{subsec::greater4_even_odd}

\begin{lemma}\label{lem::greater4_even_odd}
Assume the same notations as in Theorem \ref{thm::boundary_arm_greater4}. Suppose that (\ref{eqn::boundary_arm_greater4_even}) holds for $2n$ with $n\ge 1$, then (\ref{eqn::boundary_arm_greater4_odd}) holds for $2n+1$.
\end{lemma}

\begin{proof}[Proof of Lemma \ref{lem::greater4_even_odd}, Upper Bound]
If $\eps \le x\le 64\eps$, by the assumption hypothesis we have
\[\PP[H_{2n+1}(\eps, x, y)]\le \PP[H_{2n}(\eps, x, y)]\lesssim \left(\frac{x}{x-y}\right)^{\alpha_{2n}^+},\]
which gives the upper bound in (\ref{eqn::boundary_arm_greater4_odd}) for $2n+1$.

In the following, we assume that $x> 64\eps$.
Let $\eta$ be an $\SLE_{\kappa}$. Define $T$ to be the first time that $\eta$ swallows $x$. For $\eps>0$, let $\tau_{\eps}$ be the first time that $\eta$ hits $B(x,\eps)$. Define $O_t$ to be the image of the rightmost point of $\eta[0,t]\cap\R$ under $g_t$. Define 
\[\hat{\tau}_{\eps}=\inf\{t: \frac{g_t(x)-O_t}{g_t'(x)}=\eps\}.\]
We stop the curve at time $\hat{\tau}_{64\eps}$. Let $\tilde{\eta}$ be the image of $\eta[\hat{\tau}_{64\eps}, \infty)$ under the centered conformal map $f:=g_{\hat{\tau}_{64\eps}}-W_{\hat{\tau}_{64\eps}}$. Then $\tilde{\eta}$ is an $\SLE_{\kappa}$. Define the event $\tilde{H}_{2n}$ for $\tilde{\eta}$.

Given $\eta[0,\hat{\tau}_{64\eps}]$, consider the event $H_{2n+1}(\eps, x, y)$. We wish to control the image of the ball $B(x,\eps)$ and the image of the half-infinite line $(-\infty, y)$ under $f$.
We have the following observations.
\begin{itemize}
\item By Koebe 1/4 theorem, we know that $\hat{\tau}_{64\eps}\le \tau_{16\eps}$. Combining with Lemma \ref{lem::image_insideball}, 
 we know that $f(B(x,\eps))$ is contained in the ball $B(f(x), 4f'(x)\eps)$.
\item At time $\hat{\tau}_{64\eps}$, there are two possibilities for the image of $y$ under $f$: if $y$ is not swallowed by $\eta[0,\hat{\tau}_{64\eps}]$, then $f(y)=g_{\hat{\tau}_{64\eps}}(y)-W_{\hat{\tau}_{64\eps}}$ is the image of $y$ under $f$; if $y$ is swallowed by $\eta[0,\hat{\tau}_{64\eps}]$, then the image of $y$ under $f$ is the image of leftmost point of $\eta[0,\hat{\tau}_{64\eps}]\cap\R$ under $f$, in this case, we still write $f(y)=g_{\hat{\tau}_{64\eps}}(y)-W_{\hat{\tau}_{64\eps}}$ as explained in Section \ref{sec::preliminaries}. 
\end{itemize}

Combining these two facts, we know that, given $\eta[0,\hat{\tau}_{64\eps}]$,  $H_{2n+1}(\eps, x, y)$ implies  $\tilde{H}_{2n}(4f'(x)\eps, f(x), f(y))$. By the assumption hypothesis, we have
\[\PP\left[H_{2n+1}(\eps, x, y)\cond \eta[0,\hat{\tau}_{64\eps}], \hat{\tau}_{64\eps}<T\right]\lesssim \left(\frac{g_{\hat{\tau}_{64\eps}}(x)-W_{\hat{\tau}_{64\eps}}}{g_{\hat{\tau}_{64\eps}}(x)-g_{\hat{\tau}_{64\eps}}(y)}\right)^{\alpha_{2n}^+}\left(\frac{g'_{\hat{\tau}_{64\eps}}(x)\eps}{g_{\hat{\tau}_{64\eps}}(x)-W_{\hat{\tau}_{64\eps}}}\right)^{\alpha_{2n-1}^+}.\]
For fixed $x$ and $y$, the quantity $g_t(x)-g_t(y)$ is increasing in $t$, thus $g_t(x)-g_t(y)\ge x-y$. Plugging in the above inequality, we have
\[\PP\left[H_{2n+1}(\eps, x, y)\right]\lesssim (x-y)^{-\alpha_{2n}^+}\eps^{\alpha_{2n-1}^+}\E\left[(g_{\hat{\tau}_{64\eps}}(x)-W_{\hat{\tau}_{64\eps}})^{\alpha_{2n}^+-\alpha_{2n-1}^+}g'_{\hat{\tau}_{64\eps}}(x)^{\alpha_{2n-1}^+} 1_{\{\hat{\tau}_{64\eps}<T\}}\right].\]
By Proposition \ref{prop::sle_boundary_estimate} and (\ref{eqn::estimate_derivative_scale}), we have
\[\PP\left[H_{2n+1}(\eps, x, y)\right]\lesssim (x-y)^{-\alpha_{2n}^+}\eps^{\alpha_{2n-1}^+}x^{-u_1(\alpha_{2n}^+)}\eps^{u_1(\alpha_{2n}^+)+\alpha_{2n}^+-\alpha_{2n-1}^+}.\]
Note that
\begin{equation}\label{eqn::from2nto2n+1}
\alpha_{2n+1}^+=u_1(\alpha_{2n}^+)+\alpha_{2n}^+.
\end{equation}
Therefore
\[\PP\left[H_{2n+1}(\eps, x, y)\right]\lesssim\left(\frac{x}{x-y}\right)^{\alpha_{2n}^+}\left(\frac{\eps}{x}\right)^{\alpha_{2n+1}^+}\]
which completes the proof. \end{proof}

\begin{proof}[Proof of Lemma \ref{lem::greater4_even_odd}, Lower Bound]
Let $\eta$ be an $\SLE_{\kappa}$. Define $T$ to be the first time that $\eta$ swallows $x$. For $\eps>0$, let $\tau_{\eps}$ be the first time that $\eta$ hits $B(x,\eps)$. We stop the curve at time $\tau_{\eps}$. Let $\tilde{\eta}$ be the image of $\eta[\tau_{\eps}, \infty)$ under the centered conformal map $f:=g_{\tau_{\eps}}-W_{\tau_{\eps}}$. Then $\tilde{\eta}$ is an $\SLE_{\kappa}$. Define the event $\tilde{H}_{2n}$ for $\tilde{\eta}$.

Given $\eta[0,\tau_{\eps}]$, consider the event $H_{2n+1}(\eps, x, y)$. We wish to control the image of the ball $B(x,\eps)$ and the image of the half-infinite line $(-\infty, y)$ under $f$.
We have the following observations.
\begin{itemize}
\item Applying Koebe 1/4 Theorem to $f$, we know that $f(B(x,\eps))$ contains the ball $B(f(x), f'(x)\eps/4)$.
\item At time $\tau_{\eps}$, we have $f(y)=g_{\tau_{\eps}}(y)-W_{\tau_{\eps}}$. Recall that if $y$ is swallowed by $\eta[0,\tau_{\eps}]$, then $f(y)$ should be understood as the image of the leftmost point of $\eta[0,\tau_{\eps}]\cap\R$ under $f$. 
\end{itemize}
Combining these two facts, we know that, given $\eta[0,\tau_{\eps}]$, the event $H_{2n+1}(\eps, x, y)$ contains $\tilde{H}_{2n}(f'(x)\eps/4, f(x), f(y))$. By the assumption hypothesis, we have
\begin{equation}\label{eqn::greater4_even_odd_lower_1}
\PP\left[H_{2n+1}(\eps, x, y)\cond \eta[0,\tau_{\eps}], \tau_{\eps}<T\right]\gtrsim \left(\frac{g_{\tau_{\eps}}(x)-W_{\tau_{\eps}}}{g_{\tau_{\eps}}(x)-g_{\tau_{\eps}}(y)}\right)^{\alpha_{2n}^+}\left(\frac{g'_{\tau_{\eps}}(x)\eps}{g_{\tau_{\eps}}(x)-W_{\tau_{\eps}}}\right)^{\alpha_{2n-1}^+}.
\end{equation}
For $t\ge 0$, let $O_t$ the image of the rightmost point of $\eta[0,t]\cap\R$ under $g_t$. Set
\[\Upsilon_t=\frac{g_t(x)-O_t}{g_t'(x)},\quad J_t=\frac{g_t(x)-O_t}{g_t(x)-W_t}.\]
Define
\[M_t=g_t'(x)^{\nu(\nu+4-\kappa)/(4\kappa)}(g_t(x)-W_t)^{\nu/\kappa},\quad \text{where }\nu=\kappa(\alpha_{2n}^+-\alpha_{2n+1}^+)\le \kappa/2-4.\]
Then $M$ is a local martinagle and the law of $\eta$ weighted by $M$ becomes the law of $\SLE_{\kappa}(\nu)$ with force point $x$. By (\ref{eqn::from2nto2n+1}), we have
\[\nu(\nu+4-\kappa)/(4\kappa)=\alpha_{2n+1}^+.\]
The local martingale $M$ can be written as
\[M_t=g_t'(x)^{\alpha_{2n+1}^+}(g_t(x)-W_t)^{\alpha_{2n}^+-\alpha_{2n+1}^+}=g_t'(x)^{\alpha_{2n-1}^+}(g_t(x)-W_t)^{\alpha_{2n}^+-\alpha_{2n-1}^+}\Upsilon_t^{\alpha_{2n-1}^+-\alpha_{2n+1}^+}J_t^{\alpha_{2n+1}^+-\alpha_{2n-1}^+}.\]
At time $t=\tau_{\eps}<T$, by Koebe 1/4 Theorem, we have $\Upsilon_t\asymp \eps$.  Since $J_t\le 1$, we have
\[M_{\tau_{\eps}}\eps^{\alpha_{2n+1}^+-\alpha_{2n-1}^+}\lesssim g_{\tau_{\eps}}'(x)^{\alpha_{2n-1}^+}(g_{\tau_{\eps}}(x)-W_{\tau_{\eps}})^{\alpha_{2n}^+-\alpha_{2n-1}^+}.\]
Combining with (\ref{eqn::greater4_even_odd_lower_1}) and $M_0=x^{\alpha_{2n}^+-\alpha_{2n+1}^+}$, we have
\[\PP[H_{2n+1}(\eps, x, y)]\gtrsim \eps^{\alpha_{2n+1}^+}x^{\alpha_{2n}^+-\alpha_{2n+1}^+}
\E^*\left[(g^*_{\tau^*_{\eps}}(x)-g^*_{\tau^*_{\eps}}(y))^{-\alpha_{2n}^+}1_{\{\tau^*<T^*\}}\right],\]
where $\PP^*$ denotes the law of $\SLE_{\kappa}(\nu)$ with force point $x$ and $g^*, \tau_{\eps}^*, T^*$ are defined for $\eta^*$ whose law is $\PP^*$ accordingly. Since $\nu\le\kappa/2-4$, the curve accumulates at the point $x$ at almost surely finite time $T^*$, thus $\{\tau^*_{\eps}<T^*\}$ always holds.
To complete the proof, it is sufficient to show
\begin{equation}\label{eqn::greater4_even_odd_lower_2}
\E^*\left[\left(g^*_{\tau^*_{\eps}}(x)-g^*_{\tau^*_{\eps}}(y)\right)^{-\alpha_{2n}^+}\right]\gtrsim (x-y)^{-\alpha_{2n}^+}.
\end{equation}
Since the quantity $g^*_t(x)-g^*_t(y)$ is increasing $t$, we know that
\[x-y\le g^*_{\tau^*_{\eps}}(x)-g^*_{\tau^*_{\eps}}(y)\le g^*_{T^*}(x)-g^*_{T^*}(y).\]
Combining with Lemma \ref{lem::stochasticbound1}, we obtain (\ref{eqn::greater4_even_odd_lower_2}) and complete the proof.
\end{proof}

\subsection{Proof of Theorems \ref{thm::boundary_arm_greater4} and \ref{thm::boundary_arm_hat}}
\begin{proof}[Proof of Theorem \ref{thm::boundary_arm_greater4}]
Combining Remark \ref{rem::sle_boundary_firststep} and Lemmas \ref{lem::greater4_even_odd} and \ref{lem::greater4_odd_even} implies the conclusion.
\end{proof}

\begin{proof}[Proof of Theorem \ref{thm::boundary_arm_hat}]
We have the following observations.
\begin{itemize}
\item By Remark \ref{rem::hat_initial}, we know that (\ref{eqn::boundary_arm_hat_odd}) holds for $n=1$.
\item By the same arguments in Section \ref{subsec::greater4_even_odd}, we could prove that, assume (\ref{eqn::boundary_arm_hat_odd}) holds for $2n-1$ with $n\ge 1$, then (\ref{eqn::boundary_arm_hat_even}) holds for $2n$ where (\ref{eqn::from2nto2n+1}) should be replaced by
\[\hat{\alpha}_{2n}^+=u_1(\hat{\alpha}_{2n-1}^+)+\hat{\alpha}_{2n-1}^+.\]
\item By the same arguments in Section \ref{subsec::greater4_odd_even}, we could prove that, assume (\ref{eqn::boundary_arm_hat_even}) holds for $2n$ with $n\ge 1$, then (\ref{eqn::boundary_arm_hat_odd}) holds for $2n+1$ where (\ref{eqn::from2n-1to2n}) should be replaced by
\[\hat{\alpha}_{2n+1}^+=u_2(\hat{\alpha}_{2n}^+)+\hat{\alpha}_{2n}^+.\]
\end{itemize}
Combining these three facts, we obtain the conclusion.
\end{proof}

\section{Boundary Arm Exponents for $\kappa\le 4$}\label{sec::boundary_less4}
\subsection{Definitions and Statements}
In this section, we assume $\kappa\in(0,4]$, let $\eta$ be a chordal SLE$_\kappa$ curve, and let $g_t$ be the corresponding Loewner maps. Since $\eta$ does not hit the boundary other than its end points, $H_n$ and $\hat H_n$ defined in Section \ref{sec::introduction} are empty sets. So we need to modify their definitions.

For $y\in\R$ and $r>0$, we define half strips:
$$L_{y;r}^-=\{z\in\HH: \Im z\le r; \Re z\le y\},\quad L_{y;r}^+=\{z\in\HH: \Im z\le r; \Re z\ge y\};$$
and write $L_y^\pm=L_{y;\pi}^\pm$.

A crosscut in a domain $D$ is an open simple curve in $D$, whose end points approach boundary points of $D$. Suppose $S$ is a relatively closed subset of $\HH$ such that $\pa S\cap \HH$ is a crosscut of $\HH$. Then we use $\pa_{\HH}^+ S$ (resp.\ $\pa_{\HH}^-S$) to denote the curve $\pa S \cap \HH$ oriented so that $S$ lies to the left (resp.\ right) of the curve. For example, $\pa_{\HH}^- L_{y;r}^-$ is from $y$ to $\infty$; and for $x\in\R$, $\pa_{\HH}^+ B(x,r)$ is from $x-r$ to $x+r$.

Let $\xi_j:[0,T_j]\to\C$, $j=-1,1$, and $\eta:[0,T)\to\C$ be three continuous curves. For $j=-1,1$, define increasing functions $R_j(t)=\max(\{0\}\cup\{s\in[0,T_j]:\xi_{j}(s)\in\eta([0,t])\})$ for $t\in[0,T)$. Let $\tau_0=0$. After $\tau_n$ is defined for some $n\ge 0$, we define $\tau_{n+1}=\inf\{t\ge \tau_n:\eta(t)\in \xi_{(-1)^{n+1}}((R_{(-1)^{n+1}}(\tau_n),T_{(-1)^{n+1}}))\}$, where we set $\inf\emptyset=\infty$ by convention, and if any $\tau_{n_0}=\infty$, then $\tau_n=\infty$ for all $n\ge n_0$.

\begin{definition}
  If $\tau_{n_0}<\infty$ for some $n_0\in\N$, then we say that $\eta$ makes (at least) $n_0$ well-oriented $(\xi_{-1},\xi_1)$-crossings.
\end{definition}

\begin{remark}
  The above name comes from the fact that the orientation-preserving reparametrizations of $\xi_1,\xi_{-1},\eta$ do not affect the event.
\end{remark}

\begin{definition} Let $x>y$, $x>0$, and $\eps>0$. Let $\eta$ be an SLE$_\kappa$ in $\HH$ from $0$ to $\infty$.
 Define $H_{2n-1}^\pi(\eps, x, y)$ to be the event that $\eta$ makes at least $(2n-1)$ well-oriented $(\pa_{\HH}^+B({x,\eps}),\pa_{\HH}^- L_y^-)$-crossings.
Define $H_{2n}^\pi (\eps, x, y)$ to be the event that $\eta$ makes at least $2n$ well-oriented $(\pa_{\HH}^- L_y^-,\pa_{\HH}^+B({x,\eps}))$-crossings.  Note that in either event, the last visit that counts is at the half circle $\pa_{\HH}^+B({x,\eps})$.
\end{definition}

The theorem below is our main theorem for $\kappa\le 4$. The function $\phi$ will be defined later in (\ref{phi}), and $\phi^{(k)}$ is the $k$ times iteration of $\phi$. 
The following estimate is useful to have a sense of $\phi^{(k)}$:
\BGE \phi^{(k)}(x)\ge \frac x2,\quad\mbox{if }x\ge 6k+3.\label{phi^k>}\EDE

\begin{theorem}\label{thm::boundary_arm_alternative}
Let $\alpha_{2n}^+$ and $\alpha_{2n-1}^+$ be defined by (\ref{eqn::boundary_arm_formula_less8}). We have the following facts.

(i) If $(\eps,x,y)$ satisfy $2^{5n-4}\eps< \phi^{(2n-2)}(x-y)$, then
\BGE \PP\left[H_{2n-1}^\pi(\eps, x, y)\right]\lesssim \frac{x^{\alpha_{2n-2}^+-\alpha_{2n-1}^+}\eps^{\alpha_{2n-1}^+}}{\prod_{j=1}^{n-1}\phi^{(2n-2j-1)} (x-y)^{\alpha_{2j}^+-\alpha_{2j-2}^+}}.\label{eqn::boundary_arm_rho_odd_upper}\EDE
If $(\eps,x,y)$ satisfy $2^{5n-1}\eps< \phi^{(2n-1)}(x-y)$, and $\eps\le x$, then
\BGE \PP\left[H_{2n}^\pi(\eps, x, y)\right]\lesssim \frac{   x^{\alpha_{2n}^+-\alpha_{2n-1}^+}\eps^{\alpha_{2n-1}^+}}
{\prod_{j=1}^{n}\phi^{(2n-2j)} (x-y)^{\alpha_{2j}^+-\alpha_{2j-2}^+}}.\label{eqn::boundary_arm_rho_even_upper}\EDE
Here the implicit constants depend only on $\kappa,n$.

(ii) For any $R>0$ and $n\in\N$, there is a constant $C_{n,R}$ depending only on $\kappa, n,R$ such that
\begin{align}
\PP\left[H_{2n-1}^\pi(\eps, x, y)\right]\ge C_{2n-1,R} {x^{\alpha_{2n-2}^+-\alpha_{2n-1}^+}\eps^{\alpha_{2n-1}^+}} , &\quad \text{provided } \eps<x,\text {and }\eps<x-y\le R,\label{eqn::boundary_arm_rho_odd_lower}\\
\PP\left[H_{2n}^\pi(\eps, x, y)\right]\ge C_{2n,R}{ x^{\alpha_{2n}^+-\alpha_{2n-1}^+}\eps^{\alpha_{2n-1}^+}},
&\quad \text{provided } \eps<x\le x-y\le R.\label{eqn::boundary_arm_rho_even_lower}
\end{align}
\end{theorem}

\begin{remark}
  Using (\ref{phi^k>}), we see that, if $x-y\ge 12n$ and $2^{5n}\eps<x-y$, then
  $$\PP\left[H_{2n-1}^\pi(\eps, x, y)\right]\lesssim \left(\frac{x}{x-y}\right)^{\alpha_{2n-2}^+}\left(\frac{\eps}{x}\right)^{\alpha_{2n-1}^+}$$
  and
 $$ \PP[H_{2n}^\pi(\eps, x, y)]\lesssim \left(\frac{x}{x-y}\right)^{\alpha_{2n}^+}\left(\frac{\eps}{x}\right)^{\alpha_{2n-1}^+}.$$
 So we get the same upper bound as in the case $\kappa>4$.
\end{remark}

\subsection{Comparison principle for well-oriented crossings} \label{subsection::5}

Let $D$ be a simply connected domain. We say that $\eta:[0,T)\to\lin D$ is a non-self-crossing curve in $D$ if $\eta(0)\in\pa D$, and for any $t_0\ge 0$, there is a unique connected component $D_{t_0}$ of $D\sem \eta[0,t_0]$ such that
$\eta(t_0+\cdot)$  is the image of a continuous curve in $\lin\U$ under a continuous map from $\lin\U$ onto $\lin{D_{t_0}}$, which is an extension of a conformal map from $\U$ onto $D_{t_0}$. For example, an SLE curve is almost surely a non-self-crossing curve.

\begin{figure}[ht!]
\begin{center}
\includegraphics[width=0.55\textwidth]{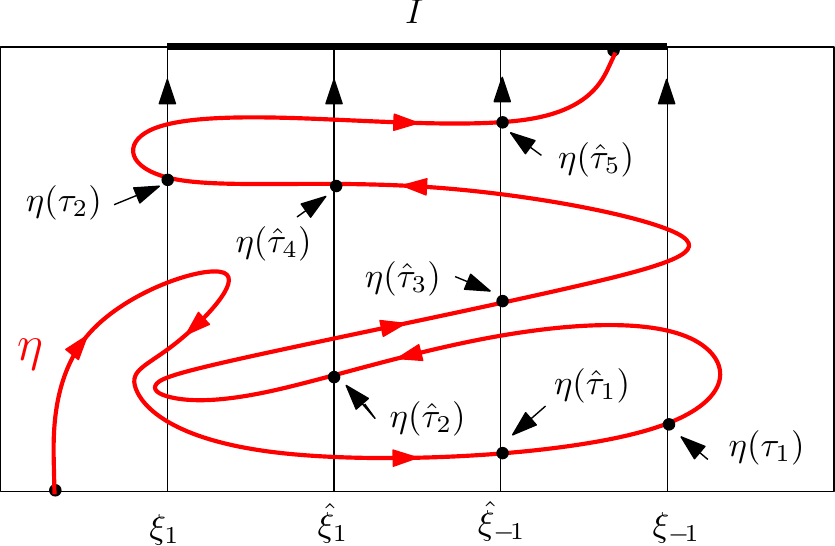}
\end{center}
\caption{\label{fig::well_oriented} The figure illustrates the definition of well-oriented crossings as well as the conditions of Lemma \ref{lem::comparison}. The curve $\eta$ totally makes $2$ well-oriented $(\xi_{-1},\xi_1)$-crossings and $5$ well-oriented $(\hat \xi_{-1}, \hat \xi_1)$-crossings. The times $\tau_j$, $1\le j\le 2$, and $\hat\tau_j$, $1\le j\le 5$, are indicated in the figure.}
\end{figure}

\begin{lemma} [Comparison Principle]
Let $D$ be a simply connected domain, and $\eta$ be a non-self-crossing curve in $D$. Let $\xi_j,\hat \xi_j:(0,1)\to \lin D$, $j=-1,1$, be crosscuts of $D$. Let $(\tau_n)$ and $R_j(t)$, $j=-1,1$   be  as in the definition of oriented crossings  for $\eta$ and $(\xi_{-1},\xi_1)$. Let $(\hat \tau_n)$ and $\hat R_j(t)$, $j=-1,1$, be the corresponding quantities for $\eta$ and $(\hat \xi_{-1},\hat \xi_1)$.
 Assume the following. See Figure \ref{fig::well_oriented}.
 \begin{enumerate}
   \item [(i)] For $j=-1,1$, $\hat\xi_j$ disconnects $\xi_j$ from both $\xi_{-j}$ and $\hat\xi_{-j}$ in $D$; the distance between $\hat\xi_{-1}$ and $\hat\xi_1$ is positive; and $\hat \xi_{-1}$ disconnects $\xi_{-1}$ from $\eta(0)$ in $D$. Here we allow the possibility that $\hat\xi_j$ touches $\xi_j$, or $\eta(0)\in\hat \xi_{-1}$.
   \item [(ii)] If $\eta_{t_0}=\hat \xi_{(-1)^{n+1}}(\hat R_{(-1)^{n+1}}(\tau_n))$ or $\hat \xi_{(-1)^{n+1}}(1)$ for some $t_0\ge \tau_n$, then for any $\eps>0$, there is $t_1\in[t_0,t_0+\eps)$ such that $\eta(t_1)\in \hat \xi_{(-1)^{n+1}}((\hat R_{(-1)^{n+1}}(\tau_n),1))$.
   \item [(iii)] There is a closed boundary (prime end) arc $I$ of $D$ with end points $\xi_1(1)$ and $ \xi_{-1}(1)$ such that $\hat \xi_j(1)\in I$, $j=-1,1$, and $\eta\cap I=\emptyset$.
 \end{enumerate}
If $\eta$ makes $n_0$ well-oriented $(\xi_{-1},\xi_1)$-crossings, then it also makes $n_0$ well-oriented $(\hat \xi_{-1}, \hat \xi_1)$-crossings. \label{lem::comparison}
\end{lemma}

\begin{remark}
  The assumption that $\eta$ is non-self-crossing forces $\eta(\tau_n+\cdot)$ to stay in the closure of the remaining domain $D_{\tau_n}$. We need assumption (iii) to prevent $\eta(\tau_n+\cdot)$ to sneak into the region bounded by the  crosscut $\hat \xi_{(-1)^{n+1}}((\hat R_{(-1)^{n+1}}(\tau_n),1))$ of $D_{\tau_n}$ through one of its endpoints without hitting the crosscut. This assumption is certainly satisfied if $\eta$ is an SLE curve.
\end{remark}

\begin{proof} Suppose $\eta$ makes $n_0$ well-oriented $( \xi_{-1}, \xi_1)$-crossings. Then $  \tau_{n_0}<\infty$. We will show that $\hat \tau_n\le \tau_n$ for $0\le n\le n_0$. Especially, the inequality $\hat \tau_{n_0}<\infty$ is what we need.

First, we have $\tau_0=\hat\tau_0=\hat R_{-1}(0)=0$. From assumptions (i) and (ii), we have
$$\hat \tau_1=\inf\{t\ge 0: \eta(t)\in \hat \xi_{-1}((0,1))\}\le \inf\{t\ge 0: \eta(t)\in  \xi_{-1}((0,1))\} = \tau_1.$$

Suppose we have proved that $\hat\tau_n\le \tau_n$ for some $n\in\{1,\dots,n_0-1\}$. Then $\eta( \tau_n)\in   \xi_{(-1)^n}$, and for every $\eps>0$, there is $t\in[ \tau_{n+1}, \tau_{n+1}+\eps)$ such that $\eta(t)\in   \xi_{(-1)^{n+1}}((  R_{(-1)^{n+1}}(  \tau_n),1))$.
Let $D_{ \tau_n}$ be the connected component of $D\sem \eta([0, \tau_n])$ such that $\eta[ \tau_n,\infty)\subset \lin{D_{ \tau_n}}$. Then $  \xi_{(-1)^{n+1}}((R_{(-1)^{n+1}}( \tau_n),1))$ is a crosscut of $D_{ \tau_n}$ since it belongs to $D\sem \eta([0, \tau_n])$ and is visited by $\eta$ after $\tau_n$. From assumption (iii) we know that $\hat \xi_{(-1)^{n+1}}((\hat R_{(-1)^{n+1}}( \tau_n),1))$ is also a crosscut of $D_{ \tau_n}$. Since $D_{\tau_n}$ is simply connected, this crosscut disconnects $ \xi_{(-1)^{n+1}}((R_{(-1)^{n+1}}(  \tau_n),1))$ from $\eta_{\tau_n}$ in $D_{\hat\tau_n}$.
From assumption (ii), we have
$$\inf\{t\ge  \tau_n: \eta(t)\in \hat \xi_{(-1)^{n+1}}((\hat R_{(-1)^{n+1}}(  \tau_n),1))\}\le \inf\{t\ge  \tau_n: \eta(t)\in  \xi_{(-1)^{n+1}}((\hat R_{(-1)^{n+1}}(  \tau_n),1))\}= \tau_{n+1}.$$
Since $\hat\tau_n\le \tau_n$ and $\hat R_{(-1)^{n+1}}(t)$ is increasing, we get $\hat R_{(-1)^{n+1}}(  \hat \tau_n)\le \hat R_{(-1)^{n+1}}(  \tau_n)$, and so
$$\hat\tau_{n+1}=\inf\{t\ge  \tau_n: \eta(t)\in \hat \xi_{(-1)^{n+1}}((\hat R_{(-1)^{n+1}}( \hat \tau_n),1))\}\le \inf\{t\ge  \tau_n: \eta(t)\in \hat \xi_{(-1)^{n+1}}((\hat R_{(-1)^{n+1}}(  \tau_n),1))\}\le  \tau_{n+1}.$$
By induction, we conclude that $\hat\tau_n\le \tau_n$ for all $0\le n\le n_0$, as desired.
\end{proof}

\begin{remark}
  The lemma also holds if we do not assume that $\xi_{-1}$ and $\hat \xi_{-1}$ are crosscuts of $D$, but assume that they are the same curve in $\lin D$.
\end{remark}

\subsection{Estimates on half strips} \label{subsection::1}

Given a nonempty $\HH$-hull $K$, Let $a_K=\min(\lin K\cap\R)$ and $b_K=\max(\lin K\cap\R)$. Let $K^{\doub}=K\cup[a_K,b_K]\cup\{\lin z: z\in K\}$. By Schwarz reflection principle, $g_K$ extends to a conformal map from $\C\sem K^{\doub}$ onto $\C\sem[c_K,d_K]$ for some $c_K<d_K\in\R$, and satisfies $g_K(\lin z)=\lin{g_K(z)}$. From  \cite[(5.1)]{ZhanLERW} we know that there is a positive measure $\mu_K$ supported by $[c_K,d_K]$ with total mass $|\mu_K|=\hcap(K)$ such that,
  \BGE f_K(z)-z=\int \frac{-1}{z-x}d\mu_K(x),\quad z\in\C\sem[c_K,d_K].\label{f-z}\EDE
  
For $x_0\in\R$ and $r>0$, let $\lin B^+(x_0,r)$ denote the special $\HH$-hull $\lin {B(x_0,r)}\cap \HH$. If an $\HH$-hull $K$ is contained in $\lin B^+(x_0,r)$, then $\hcap(K)\le \hcap(\lin B^+(x_0,r))=r^2$ by the monotonicity of half-plane capacity, and $[c_K,d_K]\subset [c_{\lin B^+(x_0,r)},d_{\lin B^+(x_0,r)}]=[x_0-2r,x_0+2r]$ by \cite[Lemma 5.3]{ZhanLERW}.

\begin{lemma}
  Let $x_0,y\in\R$ and $R,r>0$. Suppose $K$ is an $\HH$-hull and $K\subset \lin B^+_{x_0,R}$. Then the unbounded connected component of $g_K(L^-_{y;r}\sem K)$ contains $L^-_{y';r'}$ for $y'=\min\{x_0-2R-\frac{2R^2}{r},y-\frac r2\}$ and $r'=r/2$. \label{lem:L-lower}
\end{lemma}
\begin{proof}
  Let $z\in L^-_{y';r'}$. Since $\Re z\le x_0-2R-\frac{2R^2}{r}$ and $[c_K,d_K]\subset[x_0-2R,x_0+2R]$, we have $|z-x|\ge \frac{2R^2}r$ for any $x\in [c_K,d_K]$. From (\ref{f-z}) and $|\mu_K|=\hcap(K)\le R^2$, we get $|f_K(z)-z|\le \frac r2$. Since $\Re z\le y'\le y-\frac r2$, we get $\Re f_K(z)\le y$. Since $0<\Im z\le r'=r/2$, we get $0<\Im f_K(z)\le r$ ($f_K$ maps $\HH$ into $\HH$). Thus, we conclude that $f_K(L^-_{y';r'})\subset L^-_{y;r}$. Since $f_K(L^-_{y';r'})$ is an unbounded domain contained in $\HH\sem K$, and $g_K=f_K^{-1}$, we get the conclusion.
\end{proof}

Now $L^-_{y;r}$ is not an $\HH$-hull since it is not bounded. But we will still find a conformal map from $\HH$ onto $\HH\sem L^-_{y;r}$. By scaling and translation, it suffices to consider $L^-_0=L^-_{0;\pi}$. We will use the map $f_{(0,i]}(z)=\sqrt{z^2-1}$ for the half open  line segment $(0,i]$, and the map $f_{\lin B^+(0,1)}$ for the unit semi-disc. Recall that $f_{\lin B^+(0,1)}^{-1}(z)=g_{\lin B^+(0,1)}(z)=z+\frac 1z$.

\begin{lemma}
  Let $f_{L^-_0}(z)=f_{(0,i]}(z)+\log(f_{\lin B^+(0,1)}(2z))$, where the branch of $\log$ is chosen so that it maps $\HH$ onto $\{0<\Im z<\pi\}$. Then $f_{L^-_0}$ maps $\HH$ conformally onto $\HH\sem L^-_0$, and satisfies $f_{L^-_0}(z)=z+\log(2z)+O(1/z)$ as $z\to\infty$, and $ f_{L^-_0}(1)=0$, $ f_{L^-_0}(-1)=\pi i$. \label{fL}
\end{lemma}
\begin{proof}
  We observe that   $z\mapsto \log(f_{\lin B^+(0,1)}(2z))$ is a conformal map from $\HH$ onto $L^+_0$, which takes $1$ and $-1$ to $0$ and $\pi i$ respectively; and $f_{(0,i]}$ is a conformal map from $\HH$ onto $\HH\sem (0,i]$, which takes both $1$ and $-1$ to $0$. So the $f_{L^-_0}$ defined by the lemma satisfies $ f_{L^-_0}(1)=0$,$ f_{L^-_0}(-1)=\pi i$. As $z\to \infty$, $f_{(0,i]}(z)=z+O(1/z)$ and $f_{\lin B^+(0,1)}(2z)=2z+O(1/z)$. So $\log(f_{\lin B^+(0,1)}(2z))=\log(2z)+O(1/z^2)$ as $z\to \infty$. Thus, $f_{L^-_0}(z)=z+\log(2z)+O(1/z)$ as $z\to \infty$.

  It remains to show that $f_{L^-_0}$ maps $\HH$ conformally onto $\HH\sem L^-_0$. It is easy to see that $f_{L^-_0}$ maps $(1,\infty)$ into $(0,\infty)$. By Schwarz-Christoffel transformation, it suffices to show that $f_{L^-_0}'(z)=\sqrt{\frac{z+1}{z-1}}$. Let $g(z)=g_{\lin B^+(0,1)}(z)/2=\frac{z}2+\frac 1{2z}$ and $f=g^{-1}$. Then $\log(f_{\lin B^+(0,1)}(2z))=\log(f(z))$. We find that $\sqrt{g(z)^2-1}=\frac{z}2-\frac 1{2z}$ and $g'(z)=\frac 12-\frac 1{2z^2}$. So $\sqrt{g(z)^2-1}=z g'(z) =\frac{f(g(z))}{f'(g(z))}$, which implies that $\frac{f'(w)}{f(w)}=\sqrt{\frac{1}{w^2-1}}$. From this we get $\frac d{dz} \log(f_{\lin B^+(0,1)}(2z))=\frac{f'(z)}{f(z)}=\frac{1}{\sqrt{z^2-1}}$. Since  $f_{(0,i]}'(z)=\frac{z}{\sqrt{z^2-1}}$, we have $f_{L^-_0}'(z)=\frac{z}{\sqrt{z^2-1}}+\frac{1}{\sqrt{z^2-1}}=\sqrt{\frac{z+1}{z-1}}$, as desired.
\end{proof}

Define $f_{L^-_y}(z)=f_{L^-_0}( z-y)+y$, which maps $\HH$ conformally onto $\HH\sem L^-_y$, and let $g_{L^-_{y}}=f_{L^-_{y}}^{-1}$.
We will use $\hm(z,D;V)$ to denote the harmonic measure of $V$ in a domain $D$ seen from $z$, i.e., the probability that a planar Brownian motion started from $z\in D$ hits $V$ before $\pa D\sem V$.

\begin{lemma}
  For any $y,m\in\R$, and any boundary arc $I\subset \pa (\HH\sem L^-_{y})$, we have $\lim_{h\to\infty} h\cdot \hm(m+ih, \HH\sem L^-_{y};I)=|g_{L^-_{y}}(I)|/\pi$, where  $|\cdot|$ is the Lebesgue measure on $\R$.
  \label{hm-infty}
\end{lemma}
\begin{proof}
  From conformal invariance of the harmonic measure, we have $$\hm(m+ih, \HH\sem L^-_{y};I)=\hm(g_{L^-_{y}}(m+ih),\HH;g_{L^-_{y}}(I).$$
  Since $|f_{L^-_{y}}(z)-z|/|z|\to 0$ as $|z|\to\infty$, we get $|g_{L^-_{y}}(z)-z|/|z|\to 0$ as $|z|\to\infty$. From this we get $$\lim_{h\to \infty}\hm(g_{L^-_{y}}(m+ih),\HH;g_{L^-_{y}}(I))/\hm(m+ih,\HH;g_{L^-_{y}}(I))=1.$$ Since $\lim_{h\to\infty} h\cdot \hm(m+ih,\HH;g_{L^-_{y}}(I))=|g_{L^-_{y}}(I)|/\pi$,  the proof is now finished.
\end{proof}

We will use $\hm(\infty,\HH\sem L^-_{y}; I)$ to denote  $\lim_{h\to\infty}\pi\cdot h\cdot \hm(m+ih, \HH\sem L^-_{y};I)$, which equals $|g_{L^-_{y}}(I)| $ by the above lemma. For example, we have $\hm(\infty,\HH\sem L^-_y; [y,y+i\pi])=2$, and $$\hm(\infty,\HH\sem L^-_y; [y,y'])=g_{L^-_y}(y')-g_{L^-_y}(y)=g_{L^-_0}(y'-y)-1,\quad y'\ge y.$$
Note that $x\mapsto f_{L^-_0}(g_{L^-_0}(x)-2)$ is a homeomorphism from $[f_{L^-_0}(3),\infty)$ onto $[0,\infty)$. Now we
define
\BGE  \phi (x)= \left\{\begin{array}{ll} f_{L^-_0}(g_{L^-_0}(x)-2 ),  & \mbox{if }x\ge f_{L^-_0}(3) ;\\
  0, &\mbox{if } x\le f_{L^-_0}(3).
\end{array}\right. \label{phi}
 \EDE

\begin{lemma}
  Let $x_0,y_0\in\R$. Let $K$ be an $\HH$-hull such that $x_0>b_K=\max(\lin K\cap\R)$. Let $\gamma$ denote the unbounded component of $\pa L^-_{y_0}\sem (\R\cup K)$. If $x_0-y_0> f_{L^-_0}(3)$, then there is $y_1\in\R$ such that $g_K(\gamma)\subset L^-_{y_1}$ and $g_K(x_0)-y_1\ge  \phi(x_0-y_0)$. \label{g(x0-y0)}
\end{lemma}
\begin{proof}
Let $L$ be the unbounded component of $  L^-_{y_0}\sem K$.  Let $y_1=\sup \Re (g_K(\gamma))$. From (\ref{f-z}) we see that $g_K=f_K^{-1}$ decreases the imaginary part of points in $\HH$. So we have $g_K(\gamma)\subset  L^-_{y_1}$. 

Let  $x_1=g_K(x_0)$. First, we prove that $x_1> y_1$. Choose $z_1\in\lin{g_K(\gamma)}$ such that $y_1=\Re z_1$.
 Suppose $x_1\le y_1$. Then $z_1\not\in\R$ for otherwise $z_1$ is the image of $\lin{\gamma}\cap \pa K$ under $g_K$, which must lie to the left of the image of $x_0$.
  Let $\gamma_v$ denote the vertical open line segment $(y_1,z_1)$. It disconnects $x_1$ from $\infty$ in  $\HH\sem g_K(L)$. Thus, $f_K(\gamma_v)$ is a crosscut in $\HH\sem(K\cup L)$, which connects $f_K(z_1)\in\gamma$ with $f_K(y_1)\ge x_0$, and separates $x_0=f_K(x_1)$ from $\infty$ in $\HH\sem(K\cup L)$.  Then for big $h>0$,
  \BGE \hm(ih,\HH\sem (K\cup L);f_K(\gamma_v))=\hm(ih,\HH\sem L;f_K(\gamma_v))
  \ge \hm(ih,\HH\sem L^-_{y_0};f_K(\gamma_v))\ge \hm(ih,\HH\sem L^-_{y_0};[y_0,x_0]).\label{[y_0x_0]}\EDE
  Here the equality holds because $f_K(\gamma_v)$ disconnects $K$ from $\infty$ in $\HH\sem L$ (here we use the fact that $L$ is the unbounded component of $L^-_{y_0}\sem K$);  the first inequality holds because $\HH\sem L^-_{y_0}\subset \HH\sem L$; and the second inequality holds because $f_K(\gamma_v)$ disconnects $[y_0,x_0]$ from $\infty$ in $\HH\sem L^-_{y_0}$.

  From conformal invariance of harmonic measure, $\HH\sem g_K(L)\supset \HH\sem L^-_{y_1}$, and $\gamma_v\subset[y_1,y_1+i\pi]$, we have
  $$\hm(ih,\HH\sem (K\cup L);f_K(\gamma_v))=\hm(g_K(ih),\HH\sem g_K(L);\gamma_v)\le \hm(g_K(ih),\HH\sem L^-_{y_1};[y_1,y_1+i\pi]).$$
  Thus, $$ \hm(ih,\HH\sem L^-_{y_0};[y_0,x_0])\le \hm(g_K(ih),\HH\sem L^-_{y_1};[y_1,y_1+i\pi]).$$
  Combining the above inequalities with (\ref{[y_0x_0]}) and letting $h\to\infty$, we get
$$\hm(\infty, \HH\sem L^-_{y_0};[y_0,x_0])\le  \hm(\infty,\HH\sem L^-_{y_1};[y_1,y_1+i\pi]).$$
Then we get $g_{L^-_0}(x_0-y_0)-1\le 2$, which contradicts that $x_0-y_0> f_{L^-_0}(3)$. Thus, $g_K(x_0)=x_1>y_1$.

Finally, since $f_K([y_1,z_1]\cup [y_1,x_1])$ disconnects $K$ from $\infty$ in $\HH\sem L$, and disconnects $[y_0,x_0]$ from $\infty$ in $\HH\sem L^-_{y_0}$, we get
$$ \hm(\infty, \HH\sem L^-_{y_0};[y_0,x_0])\le  \hm(\infty,\HH\sem L^-_{y_1};[y_1,y_1+i\pi]\cup [y_1,x_1]),$$
which implies that $g_{L^-_0}(x_0-y_0)-1\le 2+g_{L^-_0}(x_1-y_1)-1$. So the proof is finished.
\end{proof}

Let $K_t$, $0\le t\le t_0$, be chordal Loewner hulls driven by $W_t$, $0\le t\le t_0$. Recall that every $K_t$ is an $\HH$-hull with $\hcap(K_t)=2t$. From (\ref{loewner}) it is easy to see that 
\BGE \sup\{\Re z:z\in K_{t_0}\}\le \max\{W_t:0\le t\le t_0\},\quad \sup\{\Im z:z\in K_{t_0}\}\le \sqrt{4t_0}. \label{L}\EDE
From \cite[Theorem 2.6]{LawlerSchrammWernerExponent1} and \cite[Lemma 5.3]{ZhanLERW}, we know that 
\BGE W_t\in [c_{K_{t_0}},d_{K_{t_0}}],\quad 0\le t\le t_0.\label{Win}\EDE

\begin{lemma}
  Let $R=L^-_y\cap L^+_x$ for some $x<y\in\R$. Then $c_R\ge x-2$.
\end{lemma}
\begin{proof}
    Let $m=(x+y)/2$. Then $R$ is symmetric w.r.t.\ $\{\Re z=m\}$. So $g_R(m+i\pi)=m$. By conformal invariance and comparison principle of harmonic measures, for any $h>\pi$, we get
\begin{align*}
  &h\cdot\hm(g_R(m+ih),\HH;[g_R(x+i\pi),m])=h\cdot\hm(m+ih,\HH\sem R; [x+i\pi,m+i\pi])\\
  \le & h\cdot \hm(m+ih,\{\Im z>\pi\}; [x+i\pi,m+i\pi])=h\cdot \hm(m+i(h-\pi ),\HH;[x,m]).
\end{align*}
Letting $h\to\infty$, we get $m-g_R(x+i\pi)\le m-x$, and so $g_R(x+i\pi)\ge  x$. Similarly,
$$h\cdot \hm(g_R(m+ih),\HH;[g_R(x),g_R(x+i\pi)])=h\cdot \hm(m+ih,\HH\sem R; [x,x+i\pi])\le h\cdot \hm(m+ih, \HH\sem L^+_x;[x,x+i\pi]).$$
Letting $h\to \infty$, and using Lemma \ref{hm-infty} (applied to right half strips) and $(g_R(m+ih)-(m+ih))/h\to 1$ as $h\to \infty$, we get $g_R(x+i\pi)-g_R(x)\le 2$. Thus, $c_R=g_R(x)\ge g_R(x+i\pi)-2\ge x-2$.
\end{proof}

\begin{lemma} Let $t_0=\pi^2/4$.
  We have $K_{t_0}\cap L^-_y\ne\emptyset$ if $y>\min\{W_t:0\le t\le t_0\}+2$. \label{intersect_L}
\end{lemma}
\begin{proof}
  Let $l=\min\{W_t:0\le t\le t_0\}$ and $r=\max\{W_t:0\le t\le t_0\}$. From (\ref{L}), we know that $K_{t_0}\subset L^-_r$. Suppose $K_{t_0}\cap L^-_y=\emptyset$ for some $y>l+2$. Then $K_{t_0}\subset R:=L^+_y\cap L^-_r$.  From \cite[Lemma 5.3]{ZhanLERW}, we get $[c_{K_{t_0}},d_{K_{t_0}}]\subset [c_R,d_R]$.
From the above lemma, we get $c_{K_{t_0}}\ge c_R \ge y-2>l$, which contradicts (\ref{Win}). So the proof is finished.
\end{proof}

The above lemma means that, if $\min\{W_t:0\le t\le \pi^2/4\}<y-2$, and if $(W_t)$ generates a chordal Loewner curve $\eta$, then $\eta$ visits $L^-_y$ before $\frac{\pi^2}4$.

\subsection{Estimate on the derivative}
\begin{proposition}
Assume the same setup as that in Proposition \ref{prop::sle_boundary_estimate} except that (\ref{eqn::requirement_b}) is replaced by
\begin{equation}\label{eqn::requirement_b'}
4  b\ge (\lambda-b)(\kappa\lambda-\kappa b+4-\kappa).
\end{equation}
Let $\tau_\eps$ be the first time that $|\eta(t)-1|\le \eps$. Then we have
\begin{equation}\label{eqn::estimate_derivative'}
\E\left[(g_{ {\tau}_{\eps}}(1)-W_{ {\tau}_{\eps}})^{\lambda-b}g_{ {\tau}_{\eps}}'(1)^{b} 1_{\{ {\tau}_{\eps}<T_0\}}\right]\asymp\eps^{u_1(\lambda)+\lambda-b},
\end{equation}
where the constants in $\asymp$ depend only on $\kappa,\lambda, b$. \label{Prop-Euclidean}
\end{proposition}
\begin{proof} Let $X_t=(g_{ t}(1)-W_{ t})^{\lambda-b}g_{t}'(1)^{b} 1_{\{ t<T_0\}}$ and $\beta=u_1(\lambda)+\lambda-b$.
First, (\ref{eqn::requirement_b'}) implies (\ref{eqn::requirement_b}) and $\beta\ge 0$. By Proposition \ref{prop::sle_boundary_estimate}, we have
$$ \E\left[X_{\hat\tau(\eps)}1_{\{\hat\tau(\eps)<T_0\}} \right]\asymp\eps^\beta. $$
From (\ref{eqn::requirement_b'}), we straightforwardly check that $X_t$ is a super martingale using It\^o's formula. In fact, if the equality in (\ref{eqn::requirement_b'}) holds, then $X_t$ agrees with the local martingale in Lemma \ref{lem::sle_mart} with $\rho^L=0$, $x^R=1$, and $\rho^R=\kappa(\lambda-b)$. Also note that $g_t'(1)$ is decreasing. Thus, from $\hat\tau_\eps\le \tau_\eps$, we get
$$\E\left[X_{\tau(\eps)}1_{\{\tau(\eps)<T_0\}} \right]\le \E\left[X_{\hat\tau(\eps)}1_{\{\hat\tau(\eps)<T_0\}} \right]\asymp\eps^\beta. $$
To prove the reverse inequality, we
follow the proof of Proposition \ref{prop::sle_boundary_estimate}  to get
$$\E\left[X_{\tau(\eps)}1_{\{\hat\tau(\eps)<T_0\}} \right]\asymp \eps^\beta \E^*[J_{\tau_\eps}^{-\beta}]\ge \eps^\beta,$$
using $\Upsilon_{\tau_\eps}\asymp \eps$, $0<J_t\le 1$ and $\beta\ge 0$.
\end{proof}

\subsection{Proof of Theorem  \ref{thm::boundary_arm_alternative}}

\begin{proof}[Proof of Theorem \ref{thm::boundary_arm_alternative}]
From Remark \ref{rem::sle_boundary_firststep}, we have  (\ref{eqn::boundary_arm_rho_odd_upper}) and (\ref{eqn::boundary_arm_rho_odd_lower}) for $n=1$.

\no {\bf From $2n-1$ to $2n$}: Suppose   (\ref{eqn::boundary_arm_rho_odd_upper}) and (\ref{eqn::boundary_arm_rho_odd_lower}) hold. Let $\sigma$ be the hitting time at $L^-_y$.

\no{\bf upper bound}. If $y\ge 0$, then we use the estimate
$$\PP[H_{2n}^\pi(\eps,x,y)]\le \PP[H_{2n-1}^\pi(\eps,x,y)]
\lesssim \frac{x^{\alpha_{2n-2}^+-\alpha_{2n-1}^+}\eps^{\alpha_{2n-1}^+}}{\prod_{j=1}^{n-1}\phi^{(2n-2j-1)} (x-y)^{\alpha_{2j}^+-\alpha_{2j-2}^+}} \le \frac{  x^{\alpha_{2n}^+-\alpha_{2n-1}^+}\eps^{\alpha_{2n-1}^+}}
{\prod_{j=1}^{n}\phi^{(2n-2j)} (x-y)^{\alpha_{2j}^+-\alpha_{2j-2}^+}},$$
where the last inequality follows from $\phi^{(2n-2j-1)} (x-y)\ge \phi^{(2n-2j)} (x-y)$, $x\ge x-y=\phi^{(0)}(x-y)$, and $\alpha_{2j}^+\ge \alpha_{2j-2}^+$. So we get (\ref{eqn::boundary_arm_rho_even_upper}).

If $y<0$, then $\eta(\sigma)\in\pa_{\HH}^- L^-_y$, and the righthand side of $\eta[0,\sigma]$ disconnects the union of $[\Re \eta(\sigma),0]$ and the righthand side of the line segment $[\Re \eta(\sigma),\eta(\sigma)]$ in $\HH\sem [\Re \eta(\sigma),\eta(\sigma)]$. From the comparison principal and conformal invariance of harmonic measure, we get
\begin{align*}
 & \hm(\infty,\HH\sem \eta[0,\sigma];\mbox{RHS of }\eta[0,\sigma])\ge \hm(\infty,\HH\sem (\eta[0,\sigma]\cup [\Re \eta(\sigma),\eta(\sigma)]);\mbox{RHS of }\eta[0,\sigma])\\
 \ge &  \hm(\infty,\HH\sem  [\Re \eta(\sigma),\eta(\sigma)];[\Re \eta(\sigma),0]\cup \mbox{ RHS of }[\Re \eta(\sigma),\eta(\sigma)]).
\end{align*}
Since $\Re \eta(\sigma)\le y$, we get
\BGE g_\sigma(x)-W_\sigma\ge x-y.\label{g-W}\EDE
The following local martingale is similar to the one used in the proof of Lemma \ref{lem::estimate_derivative_even} (recall (\ref{eqn::from2n-1to2n})):
$$ M_t=|g_t(x+3\eps)-W_t|^{\alpha^+_{2n}-\alpha^+_{2n-1}}g_t'(x+3\eps)^{\alpha^+_{2n-1}} . $$ 
The law of $\eta$ weighted by $M_t/M_0$ is SLE$(\kappa;\nu)$ with force point at $x+3\eps$, where $\nu=\kappa(\alpha^+_{2n}-\alpha^+_{2n-1})$. Let  $\EE^*$ denote the  expectation w.r.t.\ this SLE$(\kappa;\nu)$ process. Let $\eps_1=4(g_\sigma(x+3\eps)-g_\sigma(x+\eps))$,  $x_1=g_\sigma(x+3\eps)$, and $y_1=\sup\{\Re g_\sigma(z):z\in \pa^\sigma_{\HH} L^-_y\}$, where we use $\pa^\sigma_{\HH} L^-_y$ to denote the remaining part of $\pa _{\HH}^-L^-_y$ at time $\sigma$ in the positive direction, i.e., the unbounded component of $\pa _{\HH}^-L^-_y\sem \eta[0,\sigma]$. Then $g_\sigma(\pa^\sigma_{\HH} L^-_y)\subset L^-_{y_1}$. From Lemma \ref{lem::extremallength_argument}, the $g_\sigma$-image of the remaining part of  $\pa_{\HH}^+ B(x,\eps)$ at time $\sigma$ in the positive direction (which touches $x+\eps$), denoted by $\pa_{\HH}^\sigma B(x,\eps)$ is enclosed by $\pa_{\HH}^+ B(x_1,\eps_1)$. From (\ref{g-W}), we get $$\eps_1\le 8\eps \le 2^{5n-1}\eps\le \phi^{(2n-1)}(x-y)\le x-y\le x_1-W_\sigma.$$
This means that $\pa_{\HH}^+ B(x_1,\eps_1)$ disconnects $W_\sigma$ from $g_\sigma(\pa_{\HH}^\sigma B(x,\eps))$.
From Lemma \ref{g(x0-y0)}, we have $x_1-y_1\ge \phi(x-y)\ge 2^4\eps>\eps_1$.
So we may apply Lemma \ref{lem::comparison} and use DMP of SLE to get
$$\PP[H_{2n}^\pi(\eps,x,y)|\eta[0,\sigma]]\le H_{2n-1}^\pi(\eps_1,x_1-W_\sigma,y_1-W_\sigma).$$
We assumed that $(\eps,x,y)$ satisfy $2^{5n-1}\eps<\phi^{(2n-1)}(x-y)$. Since $g_\sigma'\le 1$ on $\R\sem K_\sigma$, we have $\eps_1\le 8\eps$. So we get $$2^{5n-4}\eps_1\le 2^{5n-1}\eps<\phi^{(2n-1)}(x-y)\le \phi^{(2n-2)}(x_1-y_1).$$
This means that $(\eps_1,x_1-W_\sigma,y_1-W_\sigma)$ satisfy the conditions for (\ref{eqn::boundary_arm_rho_odd_upper}). From the induction hypothesis, we get
\begin{align*}
  &\PP[H_{2n-1}^\pi(\eps_1,x_1-W_\sigma,y_1-W_\sigma)]\lesssim f_n(x_1-y_1) (x_1-W_\sigma)^{\alpha^+_{2n-2}-\alpha^+_{2n-1}}\eps_1^{\alpha^+_{2n-1}}\\
  \le  & f_n(x_1-y_1)  (g_\sigma(x+3\eps)-W_\sigma)^{\alpha^+_{2n-2}-\alpha^+_{2n-1}}
(g_\sigma'(x+3\eps)\eps)^{\alpha^+_{2n-1}} ,
\end{align*}
where $f_n(x_1-y_1)$ is the factor coming from the denominator of (\ref{eqn::boundary_arm_rho_odd_upper}), and the last inequality follows from $0<g_\sigma(x+3\eps)-g_\sigma(x+\eps)\le g_\sigma(x+3\eps)-V_\sigma\le 3 g_\sigma'(x+3\eps)\eps$ and $\alpha^+_{2n-1}, \alpha^+_{2n-1}\ge 0$. 
So we get
\begin{align*}
  &\PP[H_{2n}^\pi(\eps,x,y)] =\EE[\PP[H_{2n}^\pi(\eps,x,y)|\eta[0,\sigma]]] \le\EE[H_{2n-1}^\pi(\eps_1,x_1-W_\sigma,y_1-W_\sigma)] \\
  \lesssim &  f_n(x_1-y_1) \eps ^{\alpha^+_{2n-1}} \EE[(g_\sigma(x+3\eps)-W_\sigma)^{\alpha^+_{2n-2}-\alpha^+_{2n-1}}\cdot g_\sigma'(x+3\eps)^{\alpha^+_{2n-1}} ]\\
  \le & f_n\circ \phi (x-y) \eps^{\alpha^+_{2n-1}}M_0 \EE^*[(g_\sigma(x+3\eps)-W_\sigma)^{\alpha^+_{2n-2}-\alpha^+_{2n}}]\\
  \le  & f_n\circ \phi (x-y) (x-y)^{\alpha^+_{2n-2}-\alpha^+_{2n}}(x+3\eps)^{\alpha^+_{2n}-\alpha^+_{2n-1}}  \eps^{\alpha^+_{2n-1}},
\end{align*}
where in the second last inequality we used $x_1-y_1\ge \phi(x-y)$, and in the last inequality we used $\alpha^+_{2n-2}\le \alpha^+_{2n-1}$ and (\ref{g-W}). Since $\eps\le x$, we get (\ref{eqn::boundary_arm_rho_even_upper}).

\no{\bf Lower bound}. We use the local martingale (similar to the one above):
$$ M_t=g_t'(x )^{\alpha^+_{2n-1}}|g_t(x )-W_t|^{\alpha^+_{2n}-\alpha^+_{2n-1}}. $$
The law of $\eta$ weighted by $M_t/M_0$ is SLE$(\kappa;\nu)$ with force point at $x$, where $\nu=\kappa(\alpha^+_{2n}-\alpha^+_{2n-1})$. Let  $\EE^*$ and $\PP^*$ denote the  expectation and probability w.r.t.\ this SLE$(\kappa;\nu)$ process.

Fix $R>1>\delta>0$ and suppose $x-y\le R$. In the proof below, we use $C$ to denote a positive constant, which depends only on $\kappa,n,R,\delta$, and may change values between lines. Let $F(\delta)$ denote the event that $\eta[0,\sigma]\subset B(0,\frac 1\delta)$, $\eta$ does not swallows $x$ at $\sigma$, and $\dist(\eta[0,\sigma],x)\ge \delta x$. Suppose $F(\delta)$ occurs. From Lemma \ref{lem:L-lower}, the image of the unbounded connected component of $L^-_y\sem \eta[0,\sigma]$ under $g_\sigma$ contains $L^-_{y_1;\frac\pi 2}$ for $y_1:=\min\{y-\frac \pi 2,-\frac 2\delta-\frac{2}{\pi\delta^2}\}$. Assume that $\eps\le \frac{\delta x}2$. From Koebe's distortion theorem, the $g_\sigma$-image of  $\pa_{\HH}^+ B(x,\eps) $ encloses $\pa_{\HH}^+ B(x_1,\eps_1) $, where $x_1=g_\sigma(x)$ and $\eps_1=\frac 49 g_\sigma'(x)\eps$. Let $x_2=2(x_1-W_\sigma)$, $y_2=2(y_1-W_\sigma)$, and $\eps_2=2\eps_1$. From DMP and scaling property of SLE and Lemma \ref{lem::comparison}, we get
$$\PP[H_{2n}^\pi(\eps,x,y)|\eta[0,\sigma],F(\delta)]\ge H^\pi_{2n-1}(\eps_2,x_2,y_2),\quad \text{if }\eps\le {\delta x}/2.$$
From \cite[(3.12)]{LawlerConformallyInvariantProcesses}, we get $|x_1-x|\le \frac 3\delta$. So we have
\BGE x_1-y_1\le \max\{x-y+\frac 3\delta+\frac\pi 2,x+\frac 2\delta+\frac{2}{\pi\delta^2}\}\le R+\frac 5{\delta^2}.\label{x1-y1}\EDE
Let $R_2=2(R+\frac 5{\delta^2})$. Then $x_2-y_2\le R_2$, and $R_2$ depends only on $R$ and $\delta$. From the induction hypothesis, on the event $F(\delta)$, we have
$$\PP[H_{2n-1}^\pi(\eps_2,x_2,y_2)]\ge C {x_2^{\alpha^+_{2n-2}-\alpha^+_{2n-1}}\eps_2^{\alpha^+_{2n-1}}} =Cg_\sigma'(x)^{\alpha^+_{2n-1}}(g_\sigma(x)-W_\sigma)^{\alpha^+_{2n-2}-\alpha^+_{2n-1}} \eps^{\alpha^+_{2n-1}}.$$
Thus, if $\eps\le \delta x/2$, then
\begin{align*}
  &\PP[H^\pi_{2n}(\eps,x,y)]\ge \EE[\PP[H^\pi_{2n}(\eps,x,y)|\eta[0,\sigma],F(\delta)]]\ge \EE[1_{F(\delta)} H^\pi_{2n-1}(\eps_2,x_2,y_2)]\\
  \ge & C\eps^{\alpha^+_{2n-1}} \EE[1_{F(\delta)} g_\sigma'(x)^{\alpha^+_{2n-1}}(g_\sigma(x)-W_\sigma)^{\alpha^+_{2n-2}-\alpha^+_{2n-1}}]\\
  =& C\eps^{\alpha^+_{2n-1}}M_0 \EE^*[1_{F(\delta)}(g_\sigma(x)-W_\sigma)^{\alpha^+_{2n-2}-\alpha^+_{2n}}]\ge C x^{\alpha^+_{2n}-\alpha^+_{2n-1}} \eps^{\alpha^+_{2n-1}}\PP^*[F(\delta)],
\end{align*}
where we used $g_\sigma(x)-W_\sigma\le x_1-y_1\le R+\frac 5{\delta^2}$ in the last inequality.

We now find some $\delta,C\in(0,1)$ depending only on $\kappa,n,R$ such that $\PP^*[F(\delta)]\ge C$. After choosing that $\delta$, the constants $C$ we had earlier also depend only on $\kappa,n,R$.
Let $\eta$ be a chordal SLE$(\kappa,\nu)$ curve started from $0$ with force point $x$, and let $W$ be the driving function. Since $\nu\ge (\frac \kappa 2-2)\vee 0$ and $x>0$, $W_t$ is stochastically bounded above by $\sqrt\kappa B_t$,  $\eta$ never swallows $x$, and $\dist(\eta[0,\infty),x)>0$. Let $E_W$ denote the event that $\min\{W_t:0\le t\le \pi^2/4\}<-R-2$ and $\max\{W_t:0\le t\le \pi^2/4\}\le R$, and let $E_B$ denote a similar event with $\sqrt\kappa B_t$ in place of $W_t$. Then the probability of $E_W$ is bounded below by the probability of $E_B$, which is bounded below by some $C_1>0$ depending only on $\kappa,R$. When $E_W$ occurs, from Lemmas \ref{L} and \ref{intersect_L}, we get $\sigma\le \pi^2/4$ and $\eta[0,\sigma]\subset [y,R]\times [0,\pi]\subset B(0,\frac 1{\delta_1})$ for $\delta_1=\frac 1{R+\pi}$. By the scaling property of SLE$(\kappa,\nu)$ curve, we see that $\dist(\eta[0,\infty),x)/x$ is a positive random variable, whose distribution depends only on $\kappa,n$ (but not on $x$). So there is $\delta_2>0$ depending only on $\kappa,n,R$ such that the probability that $\dist(\eta[0,\infty),x)\le \delta_2 x$ is at most $C_1/2$. Let $\delta=\delta_1\wedge \delta_2$ and $C=C_1/2$. Then  $\PP^*[F(\delta)]\ge C$. For such $\delta$, if $\eps\le \delta x/2$, then
$\PP[H^\pi_{2n}(\eps,x,y)]\ge C x^{\alpha^+_{2n}-\alpha^+_{2n-1}} \eps^{\alpha^+_{2n-1}}$. Finally, if $\eps\ge \delta x/2$, then by comparison principle, we have
$$\PP[H^\pi_{2n}(\eps,x,y)]\ge \PP[H^\pi_{2n}(\delta x/2,x,y)]\ge C x^{\alpha^+_{2n}}\ge C x^{\alpha^+_{2n}-\alpha^+_{2n-1}} \eps^{\alpha^+_{2n-1}},$$
where we used $\eps\le x$ and $\alpha^+_{2n-1}\ge 0$ in the last inequality. So we get (\ref{eqn::boundary_arm_rho_even_lower}) as long as $\eps\le x$.

\no{\bf From $2n$ to $2n+1$}. Suppose   (\ref{eqn::boundary_arm_rho_even_upper}) and (\ref{eqn::boundary_arm_rho_even_lower}) hold. We use the local martingale
$$ M_t=g_t'(x)^{\alpha^+_{2n+1}}(g_t(x)-W_t)^{\alpha^+_{2n}-\alpha^+_{2n+1}} =g_t'(x)^{\alpha^+_{2n-1}}(g_t(x)-W_t)^{\alpha^+_{2n}-\alpha^+_{2n-1}}\Upsilon_t^{\alpha^+_{2n-1}-\alpha^+_{2n+1}} J_t^{\alpha^+_{2n+1}-\alpha^+_{2n-1}},$$ 
which is similar to the one used in the proof of Proposition \ref{prop::sle_boundary_estimate} (recall (\ref{eqn::from2nto2n+1})).  The law of $\eta$ weighted by $M_t/M_0$ is SLE$(\kappa;\nu)$ with force point at $x$, where $\nu=\kappa(\alpha^+_{2n}-\alpha^+_{2n+1})$. Let  $\EE^*$ and $\PP^*$ denote the  expectation and probability w.r.t.\ this SLE$(\kappa;\nu)$ process. Let $\tau_r$ be the hitting time at $\pa_{\HH}^+ B(x,r)$ for any $r>0$. Recall that $\Upsilon_{\tau_r}\asymp r$.

\no{\bf Upper bound}. First, suppose $6\eps\ge x$. Then we use the estimate
$$\PP[H_{2n+1}^\pi(\eps,x,y)]\le \PP[H_{2n}^\pi(\eps,x,y)] \lesssim \frac{  x^{\alpha_{2n}^+-\alpha_{2n-1}^+}\eps^{\alpha_{2n-1}^+}}
{\prod_{j=1}^{n}\phi^{(2n-2j)} (x-y)^{\alpha_{2j}^+-\alpha_{2j-2}^+}}
\lesssim \frac{x^{\alpha_{2n}^+-\alpha_{2n+1}^+}\eps^{\alpha_{2n+1}^+}}{\prod_{j=1}^{n}\phi^{(2n-2j-1)} (x-y)^{\alpha_{2j}^+-\alpha_{2j-2}^+}},$$
where we used $\alpha_{2j}^+\ge \alpha_{2j-2}^+$, $\phi^{(2n-2j)} (x-y)\le \phi^{(2n-2j-1)} (x-y)$, $\alpha_{2n+1}^+\ge \alpha_{2n-1}^+$, and $\eps\gtrsim x$.
So we get (\ref{eqn::boundary_arm_rho_odd_upper}).

Now suppose $6\eps<x$. Let $\sigma=\tau_{6\eps}$. Then $\eta_\sigma\in \pa_{\HH}^+ B(x,6\eps)$. Let $\eps_1=g_\sigma'(x)\eps/(1-1/6)^2$, $x_1=g_\sigma(x)$, $y_1=\sup\{\Re g_\sigma(z):z\in \pa_{\HH}^\sigma L^-_y\}$, where $\pa_{\HH}^\sigma L^-_y$ is the unbounded connected component of $\pa_{\HH}^-L^-_y\sem \eta[0,\sigma]$. Then $g_\sigma(\pa_{\HH}^\sigma L^-_y)\subset L^-_{y_1}$ because $g_\sigma$ decreases the imaginary part. From Koebe's distortion theorem, the image of $\pa_{\HH}^+ B(x,\eps)$ under $g_\sigma$ is enclosed by $\pa_{\HH}^+ B({x_1,\eps_1})$.

Since the semicircle $\pa_{\HH}^+ B(x,6\eps)$ disconnects the union of $[0,x)$ and the righthand side of $\eta[0,\sigma)$ from $\infty$ in $\HH\sem\eta[0,\sigma]$, by the conformal invariance and comparison principle for harmonic measure, we have
\begin{align*}
  &\hm(\infty,\HH;[x-12\eps,x+12\eps])=\hm(\infty,\HH,\pa_{\HH}^+ B(x,6\eps))\ge \hm(\infty,\HH\sem\eta[0,\sigma];\pa_{\HH}^+ B(x,6\eps))\\
  \ge & \hm(\infty,\HH\sem\eta[0,\sigma];[0,x]\cup\mbox{ RHS of }\eta[0,\sigma])=\hm(\infty,\HH;[W_\sigma,x_1]).
\end{align*}
Thus, $x_1-W_\sigma\le 24\eps$. Since $x_1-y_1\ge \phi(x-y)\ge \phi^{(2n)}(x-y)\ge 2^{5n}\eps>24\eps$, we get $y_1-W_\sigma<0$.
This means that $\pa_{\HH}^- L^-_{y_1}$ disconnects $W_\sigma$ from $g_\sigma(\pa_{\HH}^\sigma L^-_y)$. Besides, since $g_\sigma'(x)\in(0,1)$, we have $x_1-y_1>\eps_1$. So we may apply Lemma \ref{lem::comparison} and use DMP of SLE to get
$$\PP[H^\pi_{2n+1}(\eps,x,y)|\eta[0,\sigma]]\le H^\pi_{2n}(\eps_1,x_1-W_\sigma,y_1-W_\sigma).$$

We assumed that $(\eps,x,y)$ satisfy $2^{5n}\eps <\phi^{(2n)}(x-y)$. Since $g_\sigma'\le 1$ on $\R\sem K_\sigma$, we have $\eps_1\le 4\eps$. Thus,
$$2^{5n-2}\eps_1\le 2^{5n}\eps< \phi^{(2n)}(x-y)\le \phi^{(2n-1)}(x_1-y_1).$$
From Koebe's $1/4$ theorem, we get $x_1-W_\sigma\ge 6g_\sigma'(x)\eps/4 \ge g_\sigma'(x)\eps/(1-1/6)^2=\eps_1$.
This means that $(\eps_1,x_1-W_\sigma,y_1-W_\sigma)$ satisfy the conditions for (\ref{eqn::boundary_arm_rho_even_upper}). From the induction hypothesis, we get
\begin{align*}
  &\PP[H^\pi_{2n}(\eps_1,x_1-W_\sigma,y_1-W_\sigma)]\lesssim f_n(x_1-y_1) (x_1-W_\sigma)^{\alpha^+_{2n}-\alpha^+_{2n-1}}\eps_1^{\alpha^+_{2n-1}}\\
  \asymp &  f_n(x_1-y_1)   \eps^{\alpha^+_{2n-1}} (g_\sigma(x )-W_\sigma)^{\alpha^+_{2n}-\alpha^+_{2n-1}}g_\sigma'(x )^{\alpha^+_{2n-1}},
\end{align*}
where $f_n(x_1-y_1)$ is the factor coming from the denominator of (\ref{eqn::boundary_arm_rho_even_upper}).
Thus,
\begin{align*}
  &\PP[H^\pi_{2n+1}(\eps,x,y)]=\EE[\PP[H^\pi_{2n+1}(\eps,x,y)|\eta[0,\sigma]]]\le \EE[H^\pi_{2n}(\eps_1,x_1-W_\sigma,y_1-W_\sigma)]\\
  \lesssim & f_n(x_1-y_1)   \eps ^{\alpha^+_{2n-1}} \EE[(g_\sigma(x)-W_\sigma)^{\alpha^+_{2n}-\alpha^+_{2n-1}}g_\sigma'(x)^{\alpha^+_{2n-1}}]\\
  \lesssim & f_n\circ \phi (x-y)  \eps^{\alpha^+_{2n-1}}x^{\alpha^+_{2n}-\alpha^+_{2n-1}} \eps^{u_1(\alpha^+_{2n})+\alpha^+_{2n}-\alpha^+_{2n-1}}
  = f_n\circ \phi (x-y) x^{\alpha^+_{2n}-\alpha^+_{2n-1}} \eps^{\alpha^+_{2n+1}}
\end{align*}
where we used Proposition \ref{Prop-Euclidean}, the scaling invariance  of SLE, and ((3.8)). Then we get (\ref{eqn::boundary_arm_rho_odd_upper}) for $2n+1$.

\no{\bf Lower bound}. We fix $R,\delta>0$ and suppose $x-y\le R$. In the proof below, we use $C$ to denote a positive constant, which depends only on $\kappa,n,R,\delta$, and may change values between lines. Let $\sigma=\tau_{\eps}$.   From Koebe's $1/4$ theorem, the $g_\sigma$-image of  $\pa_{\HH}^+ B(x,\eps) $ encloses  $\pa_{\HH}^+ B(x_1,\eps_1)$, where $x_1=g_\sigma(x)$ and $\eps_1=g_\sigma'(x)\eps/4$. Let $F(\delta)$ denote the event that $\sigma<\infty$, $x$ is not swallowed at $\sigma$, and $\eta[0,\sigma]\subset B(0,\frac 1\delta)$. Suppose $F(\delta)$ occurs. From Lemma \ref{lem:L-lower}, the image of the unbounded connected component of $L^-_y\sem \eta[0,\sigma]$ under $g_\sigma$ contains $L^-_{y_1;\frac\pi 2}$ for $y_1:=\min\{y-\frac \pi 2,-\frac 2\delta-\frac{2}{\pi\delta^2}\}$. Let $x_2=2(x_1-W_\sigma)$, $y_2=2(y_1-W_\sigma)$, and $\eps_2=2\eps_1$. From DMP and scaling property of SLE and Lemma \ref{lem::comparison}, we get
$$\PP[H^\pi_{2n+1}(\eps,x,y)|\eta[0,\sigma],F(\delta)]\ge H^\pi_{2n}(\eps_2,x_2,y_2).$$
Using the same argument as around (\ref{x1-y1}), we get $x_2-y_2\le R_2:=2(R+\frac 5{\delta^2})$. From the induction hypothesis, on the event $F(\delta)$, we have
$$\PP[H^\pi_{2n}(\eps_2,x_2,y_2)]\ge C {x_2^{\alpha^+_{2n}-\alpha^+_{2n-1}}\eps_2^{\alpha^+_{2n-1}}} =Cg_\sigma'(x)^{\alpha^+_{2n-1}}(g_\sigma(x)-W_\sigma)^{\alpha^+_{2n}-\alpha^+_{2n-1}} \eps^{\alpha^+_{2n-1}}.$$
Thus,
\begin{align}
  &\PP[H^\pi_{2n+1}(\eps,x,y)]\ge \EE[\PP[H^\pi_{2n+1}(\eps,x,y)|\eta[0,\sigma],F(\delta)]]\ge \EE[1_{F(\delta)} H^\pi_{2n}(\eps_2,x_2,y_2)]\nonumber\\
  \ge & C\eps^{\alpha^+_{2n-1}} \EE[1_{F(\delta)} g_\sigma'(x)^{\alpha^+_{2n-1}}(g_\sigma(x)-W_\sigma)^{\alpha^+_{2n}-\alpha^+_{2n-1}}]\nonumber\\
   =& C\eps^{\alpha^+_{2n-1}}M_0 \EE^*[1_{F(\delta)}  J_\sigma^{\alpha^+_{2n-1}-\alpha^+_{2n+1}} \Upsilon_\sigma^{\alpha^+_{2n+1}-\alpha^+_{2n-1}}]\ge C x^{\alpha^+_{2n}-\alpha^+_{2n-1}}\eps^{\alpha^+_{2n-1}} \PP^*[F(\delta)],\label{PF-del}
\end{align}
where in the last inequality we used $\Upsilon_\sigma\asymp \eps$, $J_\sigma\in(0,1]$, and $\alpha^+_{2n-1}-\alpha^+_{2n+1}\le 0$.

We now find some $\delta,C>0$ depending only on $\kappa,n,R$ such that $\PP^*[F(\delta)]\ge C$. After choosing that $\delta$, the constants $C$ we had earlier also depend only on $\kappa,n,R$. Let $\eta$ be a chordal SLE$(\kappa,\nu)$ curve started from $0$ with force point $x$. Since $\nu\le\kappa/2-4$, the curve $\eta$ goes all the way to $x$ in finite time, and so is bounded. Moreover, $\eta$ does not swallow $x$ before it reaches $x$. By scaling property, $\diam(\eta)/x$ is a bounded random variable, whose distribution depends only on $\kappa,n$. Thus, there are constants $\delta_1,C >0$ depending only on $\kappa,n$, such that $\PP^*[F(\delta_1/x)]\ge C $. Then we let $\delta=\delta_1/R$. Since $x\le x-y\le R$, we have $F(\delta_1/x)\subset F(\delta)$. Using such $\delta$ and applying (\ref{PF-del}), we get (\ref{eqn::boundary_arm_rho_odd_lower}) for $2n+1$.
\end{proof}

\bibliographystyle{alpha}
\bibliography{bibliography}
\smallbreak
\noindent Hao Wu\\
\noindent NCCR/SwissMAP, Section de Math\'{e}matiques, Universit\'{e} de Gen\`{e}ve, Switzerland\\
\noindent\textit{and}
Yau Mathematical Sciences Center, Tsinghua University, China\\
\noindent hao.wu.proba@gmail.com
\smallbreak
\noindent Dapeng Zhan\\
\noindent {Department of Mathematics, Michigan State University, United States of America}\\
\noindent zhan@math.msu.edu

\end{document}